\documentclass[oneside,english]{amsart}
\usepackage[T1]{fontenc}
\usepackage[latin9]{inputenc}
\usepackage{amstext}
\usepackage{amsthm}
\usepackage{amsmath}
\usepackage{amssymb}
\usepackage{esint}
\usepackage{enumitem}
\usepackage{soul}
\usepackage[usenames,dvipsnames]{xcolor}
\usepackage[
colorlinks=true,
linkcolor=blue,
citecolor=blue,
urlcolor=blue,
pagebackref,
]{hyperref}

\usepackage[yyyymmdd,hhmmss]{datetime}

\makeatletter
\numberwithin{equation}{section}
\numberwithin{figure}{section}
\theoremstyle{plain}

  \theoremstyle{plain}
  
  \theoremstyle{plain}
\newtheorem{theorem}{\protect\theoremname}[section]
\newtheorem*{theorem*}{Theorem}
  \theoremstyle{plain}
  \newtheorem{corollary}[theorem]{\protect\corollaryname}
  \theoremstyle{plain}
    \newtheorem{definition}[theorem]{Definition}
   
  \theoremstyle{plain}
     \newtheorem{proposition}[theorem]{Proposition}
  \theoremstyle{plain}
  \newtheorem{lemma}[theorem]{\protect\lemmaname}
  \theoremstyle{plain}
  \newtheorem{remark}[theorem]{\protect\remarkname}

    \newtheorem{example}[theorem]{Example}
        \newtheorem{question}{Question}

\DeclareMathOperator{\esssup}{ess\sup}
\DeclareMathOperator{\essinf}{ess\inf}

\newcommand{\tom}{\mathbb{T}^\omega}

\def\XXint#1#2#3{{\setbox0=\hbox{$#1{#2#3}{%
\int}$ }
\vcenter{\hbox{$#2#3$ }}\kern-.6\wd0}}

\usepackage{pgf,tikz}
\usetikzlibrary{arrows}
\usepackage{url}

\makeatother

\usepackage{babel}
  \providecommand{\lemmaname}{Lemma}
   \providecommand{\corollaryname}{Corollary}
  \providecommand{\remarkname}{Remark}
\providecommand{\theoremname}{Theorem}

\providecommand{\keywords}[1]{\textbf{\textit{Index terms---}} #1}

\allowdisplaybreaks

\begin{document}

\title[ ]{Maximal operators\\ on the infinite-dimensional torus}



\author{Dariusz Kosz, Javier C. Mart\'{i}nez-Perales, Victoria Paternostro, Ezequiel Rela, and Luz Roncal}
\address{
}
\email{}

\address[D. Kosz]{
BCAM -- Basque Center for Applied Mathematics \\
48009 Bilbao, Spain and Faculty of Pure and Applied Mathematics, Wroc\l aw University of Science and Technology \\
Wybrze\.ze Wyspia\'nskiego 27, 50-370 Wroc\l aw, Poland
}
\email{dariusz.kosz@pwr.edu.pl}

\address[J.C. Mart\'{i}nez-Perales]{
BCAM -- Basque Center for Applied Mathematics\\
48009 Bilbao, Spain
}
\email{jmartinez@bcamath.org}

\address[V. Paternostro and E. Rela]{  Departamento de Matem\'atica, Facultad de Ciencias Exactas y Naturales, Universidad de Buenos Aires, Buenos Aires, Argentina}
\email{\{vpater,erela\}@dm.uba.ar}

\address[L. Roncal]{
BCAM -- Basque Center for Applied Mathematics\\
48009 Bilbao, Spain, Ikerbasque, Basque Foundation for Science, 48011 Bilbao, Spain, and Universidad del Pa\'is Vasco / Euskal Herriko Unibertsitatea, 48080 Bilbao, Spain}
\email{lroncal@bcamath.org}

\begin{abstract}
We study maximal operators related to bases on the infinite-dimensional torus $\tom$. {For the normalized Haar measure $dx$ on $\tom$ it is known that $M^{\mathcal{R}_0}$, the maximal operator associated with the dyadic basis $\mathcal{R}_0$, is of weak type $(1,1)$, but $M^{\mathcal{R}}$, the operator associated with the natural general basis $\mathcal{R}$, is not. We extend the latter result to all $q \in [1,\infty)$. Then we find a wide class of intermediate bases $\mathcal{R}_0 \subset \mathcal{R}' \subset \mathcal{R}$, for which maximal functions have controlled, but sometimes very peculiar behavior.} Precisely, for given $q_0 \in [1, \infty)$ we construct $\mathcal{R}'$ such that $M^{\mathcal{R}'}$ is of restricted weak type $(q,q)$ if and only if $q$ belongs to a predetermined range of the form $(q_0, \infty]$ or $[q_0, \infty]$. Finally, we study the weighted setting, considering the Muckenhoupt $A_p^\mathcal{R}(\tom)$ and reverse H\"older $\mathrm{RH}_r^\mathcal{R}(\tom)$ classes of weights associated with $\mathcal{R}$. For each $p \in (1, \infty)$ and each $w \in A_p^\mathcal{R}(\tom)$ we obtain that $M^{\mathcal{R}}$ is not bounded on $L^q(w)$ in the whole range $q \in [1,\infty)$. Since we are able to show that 
\[
\bigcup_{p \in (1, \infty)}A_p^\mathcal{R}(\tom) = \bigcup_{r \in (1, \infty)} \mathrm{RH}_r^\mathcal{R}(\tom),
\]
the unboundedness result applies also to all reverse H\"older weights.
\end{abstract}

\maketitle

\begin{center}

\end{center}

\tableofcontents

\section{Introduction and main results}\label{sec:Intro}

 The  infinite dimensional torus $\tom$ is the compact abelian group consisting of countable infinite many copies of the one dimensional torus $\mathbb{T}:=\mathbb{R}/\mathbb{Z}$, understood as the interval $[0,1]$ with its endpoints identified.  We refer to Subsection~\ref{subsec:tom} for basic definitions related to $\tom$. In a recent paper, Fern\'andez and the fifth author \cite{Fernandez-Roncal2020} introduced a Calder\'on--Zygmund decomposition in $\tom$. By defining a suitable ``dyadic'' basis (the \textit{restricted Rubio de Francia} basis, denoted by $\mathcal{R}_0$), they proved that the corresponding ``dyadic'' Hardy--Littlewood maximal function  satisfies the weak type $(1,1)$ inequality. Soon thereafter, the first author \cite{Kosz2020} showed that the maximal function associated with a wider natural basis (the \textit{Rubio de Francia basis}, denoted by $\mathcal{R}$) is not of weak type $(1,1)$. See Subsection~\ref{sub:bases} for the definitions of $\mathcal{R}_0$ and $\mathcal{R}$.

Maximal operators are central tools in the theory of harmonic analysis. In the Euclidean setting, for instance, the boundedness of the maximal function is used to prove the Lebesque differentiation theorem and the ergodic theorem. In general, a precise knowledge of the mapping properties of maximal functions is crucial to develop further investigation in the corresponding setting. 

Motivated by the above facts, a first natural question arises: 
\begin{itemize}
\item What is the behavior of the maximal operators associated with bases $\mathcal{R}'$, denoted by $M^{\mathcal{R}'}$, which fall between the restricted Rubio de Francia basis $\mathcal{R}_0$ and the Rubio de Francia basis $\mathcal{R}$?
\end{itemize}

Our first main result will be as follows: Let $q \in [1, \infty)$. We can find a wide class of intermediate bases $\mathcal{R}_0 \subset \mathcal{R}' \subset \mathcal{R}$, for which  the corresponding maximal functions is not even of restricted weak type $(q,q)$. The result will be accomplished through a special construction involving what we will call $(\varepsilon,l)$\textit{-configurations}, which are defined as follows. 

\begin{definition}\label{epsn}
Let $0 < \varepsilon \leq \frac{1}{2}$ and $l \in \mathbb{N}$. A collection $\{Q^{(1)}, \dots, Q^{(l)}\} \subset \mathcal{R}$ is an $(\varepsilon, l)$-configuration if there exists a measurable
set $A \subset \tom$ such that
\begin{enumerate}
    \item $|Q^{(k)}|=|A|$ and $\varepsilon |A| \leq | Q^{(k)} \cap A| \leq (1- \varepsilon) |A|$ hold for each $k$,
    \item the sets $Q^{(k)} \setminus A$, $k \in \{1, \dots, l\}$, are disjoint. 
\end{enumerate}
\end{definition}

\begin{theorem}\label{unweighted_unboundedness}
 Let $q \in [1, \infty)$. Consider a sequence $\{l_j\}_{j\in\mathbb{N}}$ of natural numbers and a sequence $\{\varepsilon_j\}_{j\in\mathbb{N}}$ of parameters contained in $(0,1/2]$ such that $\sup_{j\in\mathbb{N}} \varepsilon_j^{q+1}l_j=\infty$. If $\mathcal{R}'=\mathcal{R}_0\cup S$, where $S$ is the union of a sequence $\{S_j\}_{j\in\mathbb{N}}$ with $S_j$ being an $(\varepsilon_j,l_j)$-configuration for all $j\in\mathbb{N}$, then $M^{\mathcal{R}'}$ is not of restricted weak type $(q,q)$. 
 \end{theorem}

In particular, from Theorem~\ref{unweighted_unboundedness} we deduce the following result for $M^{\mathcal{R}}$.

\begin{corollary}\label{corollary}
Let $q \in [1, \infty)$. The maximal operator $M^{\mathcal{R}}$ is not of restricted weak type $(q,q)$.
\end{corollary} 

Theorem \ref{unweighted_unboundedness} sheds some light  on the question of characterizing threshold basis for which the associated maximal function is unbounded. It is not easy to find the exact threshold properties, and the notion of $(\varepsilon,l)$-configuration is an important step towards this goal. The second main result exploits the strategy introduced in Theorem~\ref{unweighted_unboundedness} more efficiently to get, indeed,  a characterization. Namely, for given $q_0 \in (1, \infty)$ we find $\mathcal{R}'$ with $\mathcal{R}_0\subset \mathcal{R}'\subset \mathcal{R}$ such that the restricted weak type $(q,q)$ inequality for the associated maximal operator holds if and only if $q$ belongs to a predetermined range of the form $[q_0,\infty]$ or $(q_0, \infty]$. 
The structure used for this relies on a particular example of $(\varepsilon, l)$-configuration. Namely, we will consider \textit{$(\varepsilon,l)$-configurations around sets} $Q\in\mathcal{R}_0$, see Definition \ref{def:confaround}.

 \begin{theorem}\label{thm:r_0}
 Let $\{Q_j\}_{j\in\mathbb{N}}$ be a sequence of dyadic cubes. Let $\mathcal{R}'=\mathcal{R}_0\cup S$, where $S$ is the union of a sequence $\{S_j\}_{j\in\mathbb{N}}$ of collections of cubes in $\mathcal{R}$ such that:
 \begin{enumerate}
     \item $S_j$ is an $(\varepsilon_j,l_j)$-configuration around $Q_j$ with $0<\varepsilon_j\leq \frac{1}{2}$ and $l_j\in\mathbb{N}$ for all $j\in\mathbb{N}$.
     \item The family   $\{E_j\}_{j\in\mathbb{N}}$ of the sets $E_j:=\bigcup_{Q\in S_j}Q$ is pairwise disjoint.
 \end{enumerate}
Then $M^{\mathcal{R}'}$ is not of restricted weak type $(1,1)$ if and only if $\sup_{j \in \mathbb{N}} l_j = \infty$, and $M^{\mathcal{R}'}$ is not of restricted weak type $(q,q)$, $q \in (1, \infty)$, if and only if $\sup_{j \in \mathbb{N}} \varepsilon_j l_j^{1/q} = \infty$.
 \end{theorem}

The introduction of $(\varepsilon,l)$-configurations plays a prominent role in our results, let us explain the heuristics at this point. A natural way to look for a threshold basis, namely, a basis $\mathcal{R}'$ for which the corresponding maximal function stops being bounded, is to add cubes to the dyadic basis $\mathcal{R}_0$ and study the properties of the associated maximal function. As it will be described in Section \ref{sec:noweight}, several simple reductions rule out specific constructions of basis which have a bounded associated maximal function. It turns out that a good choice of additional cubes, in the sense that they make the new basis have an unbounded maximal function, consists of controlled non-dyadic translations of cubes of the basis $\mathcal{R}_0$, this is the content of Theorem \ref{unweighted_unboundedness}. Yet another refinement entails translations of sequences of dyadic cubes so that each dyadic cube is translated in a different direction in such a way that the resulting collection has a bounded overlapping in the $L^1(\tom)$ sense (this is reminiscent of the so-called sparse collections, see Subsection \ref{unweighted_unboundedness}). The interplay between the number of translations and the scale of the translations is fundamental in this procedure. If the number of translations is relatively large compared to the size of each translation, or the other way around, the construction leads to a basis with associated maximal function which is not bounded. This is the idea behind $(\varepsilon,l)$-configurations around sets $Q\in \mathcal{R}_0$ (see Figure~\ref{fig:el-conf}) in the characterization in Theorem \ref{thm:r_0}.

 The proof of Theorem \ref{thm:r_0} combines elements from harmonic analysis and probability, and has a combinatorial flavor. The result reveals interesting substantial differences with the finite-dimensional Euclidean setting and it bears out that the geometric nature of the infinite-dimensional torus leads to unexpected behavior of maximal functions.  
Using Theorem~\ref{thm:r_0} we obtain the following.

\begin{corollary}\label{corollary2}
\label{threshold}
Let $q_0 \in [1, \infty)$.
\begin{enumerate}
    \item It is possible to find a basis $\mathcal{R}'$ with $\mathcal{R}_0\subseteq\mathcal{R}'\subset \mathcal{R}$ such that $M^{\mathcal{R}'}$ is of restricted weak type $(q,q)$ if $q \in [q_0, \infty]$ and is not if $q \in [1, q_0)$.
    \item Similarly, it is possible to find a basis $\mathcal{R}'$ with $\mathcal{R}_0\subset\mathcal{R}'\subset \mathcal{R}$ such that $M^{\mathcal{R}'}$ is of restricted weak type $(q,q)$ if $q \in (q_0, \infty]$ and is not if $q \in [1, q_0]$.
\end{enumerate}
Note that for item (1) the choice $q_0=1$ is trivial: choose $\mathcal{R}'=\mathcal{R}_0$. Also, we know that, choosing $\mathcal{R}'=\mathcal{R}$, we get a basis for which $M^{\mathcal{R}'}$ is unbounded for all  $q\in[1,\infty)$. 

\end{corollary}

In view of the above results, one may ask 
whether the situation changes significantly if the normalized Haar measure is replaced by a different measure which is absolutely continuous with respect to it (actually this is a necessary condition if we want this new measure to satisfy some regularity properties, see \cite[Theorem~1.4]{Duo2013}). This motivates to consider weighted inequalities, the first question being:
\begin{itemize}
    \item What are the weighted boundedness properties of the maximal operator associated with the Rubio de Francia basis $\mathcal{R}$? \end{itemize}

The Muckenhoupt classes $A_p^{\mathcal{R}}(\tom)$ with $p \in [1, \infty)$ (see Subsection~\ref{subsec:maxw} for the definition) are the obvious classes of weights to study in this context. For them we again have been able to show a negative result in the whole range of $q$'s. Even more can be said, the same is true for $M^{\mathcal{R}'}$ instead of $M^{\mathcal{R}}$, where $\mathcal{R}' = \mathcal{R}_0 \cup S$ and $S$ is built with the aid of a suitable sequence of $(\varepsilon_j,l_j)$-configurations around sets $Q_j$. For example, any sequence satisfying $\sup_{j \in \mathbb{N}} l_j = \infty$ and $\varepsilon_j = \varepsilon$, where $\varepsilon$ is sufficiently small, is good for this purpose.

To quantify the restriction on $\varepsilon$, we announce that for any weight $w \in A_p^{\mathcal{R}}(\tom)$, $p \in (1, \infty)$, there exist $\delta \in (0, 1]$ and $C>0$ such that
\begin{equation}
\label{eq:condelta_new}
  \frac{w(E)}{w(Q)}\leq C\left(\frac{|E|}{|Q|}\right)^\delta
\end{equation}
holds for all $E\subset Q$ and $Q \in \mathcal{R}$ (see Theorem~\ref{teo:weights2}). Here, for a measurable set $A\subseteq \tom$, $w(A):=\int_A w(x)\,dx$.
Then, for $w$ satisfying \eqref{eq:condelta_new}, one can take any $\varepsilon \in (0, \frac{1}{2}]$ such that $C \varepsilon^\delta < 1$. Although it is not necessary, to make the argument clear the parameter $\varepsilon$ will be specified to be equal to $\frac{1}{N}$ for some large $N \in \mathbb{N}$.

\begin{theorem}\label{thm:unboundedweights_new}
 Let $w$ be a Muckenhoupt weight in $ A_p^{\mathcal{R}}(\tom)$ for some $p \in (1,\infty)$. Choose $N \in \mathbb{N}$ to be the smallest positive integer such that $C N^{-\delta} < 1$, where $C>0$ and $\delta \in (0, 1]$ are specified in \eqref{eq:condelta_new}. Let $\{Q_j\}_{j\in\mathbb{N}}$ be a sequence of dyadic cubes and $\{l_j\}_{j\in\mathbb N}$, $\{N_j\}_{j\in\mathbb N}$ be sequences of positive integers with $N_j \geq N$ for all $j\in\mathbb{N}$. Assume that $\mathcal{R}'=\mathcal{R}_0\cup S$, where $S$ is the union of a sequence $\{S_j\}_{j\in\mathbb{N}}$ with each $S_j$ being a $(\frac{1}{N_j},l_j)$-configuration around $Q_j$. If for some $q \in [1, \infty)$ we have 
 \[
 \sup_{j \in \mathbb{N}} \frac{\big( 1-CN_j^{-\delta} \big)^{N_j/q} l_j^{1/q}}{N_j} = \infty,
 \]
 then $M^{\mathcal{R}'}$ is not of restricted weak type $(q,q)$ with respect to $w$.
\end{theorem}

\noindent As an immediate corollary we obtain the following.

\begin{corollary}\label{corolweights_new}
Let $w$ be a Muckenhoupt weight in $ A_p^{\mathcal{R}}(\tom)$ for some $p \in (1,\infty)$. Then $M^\mathcal{R}$ is not of weighted restricted weak type $(q,q)$ with respect to $w$ for all $q \in [1, \infty)$.
 \end{corollary}

\noindent Using the above corollary we get that the same is true also for all reverse H\"older classes $\mathrm{RH}_r^{\mathcal{R}}(\tom)$ with $r \in (1,\infty)$ (see Subsection~\ref{subsec:maxw} for the definition).
\begin{corollary}\label{corolweights_new2}
Let $w$ be a reverse H\"older weight in $ \mathrm{RH}_r^{\mathcal{R}}(\tom)$ for some $r \in (1,\infty)$. Then $M^\mathcal{R}$ is not of weighted restricted weak type $(q,q)$ with respect to $w$ for all $q \in [1, \infty)$.
 \end{corollary}

 In view of the negative results stated above one might be tempted to look for scenarios with more positive outcomes. The positive unweighted results in Corollary~\ref{threshold} bring the attention to the intermediate bases $\mathcal{R}'$ used there. Thus, a natural next question is whether $A_p^{\mathcal{R}'}(\tom)$ is the correct class to characterize the weighted boundedness of $M^{\mathcal{R}'}$ on $L^p(w)$ for $p \in (q_0, \infty)$, at least for some of the bases $\mathcal{R}'$. However, as we shall explain in detail at the end of Section~\ref{weighted_unb}, the answer in this case looks nontrivial.

As pointed out by Rubio de Francia \cite{RubioConv}, the interest for the study of Fourier analysis on $\tom$ has several motivations. First, it can be seen as a natural extension of $n$-dimensional Fourier analysis, but with estimates uniform in the dimension. On the other hand, $\{e^{2\pi ix_k}\}_{k\in \mathbb{N}}$ is a system of independent and identically distributed random variables in the complex unit circle (i.e., a complexified version of Rademacher functions), whose natural completion on $L^2$ is the trigonometric system on $\tom$. Hence, Fourier series of infinite variables turn out to be the complex analogue of Walsh series.
Fourier series on $\tom$ have important connections with the classical theory of Dirichlet series \cite{Bohr, Toeplitz}, leading to profound problems in analytic number theory and functional analysis, see the remarkable work \cite{HLS} (also, e.g., \cite{  AOS, AOS2, Bayart, BDFMS,HS, PP}). There is also a considerable interest on the infinite torus from the point of view of potential theory \cite{Bendikov, BCS, BS}.

Finally, we would like to point out the relation of the analysis in infinite dimensions with Malliavin calculus, which is one of the main tools in modern stochastic analysis. This theory motivates the analysis on Wiener space, an infinite-dimensional generalization of the usual analytical concepts on $\mathbb{R}^n$, such as Fourier analysis and Sobolev spaces \cite{Hairer}. 

\noindent \textbf{Structure of the paper.} We start presenting preliminaries and definitions about the infinite-dimensional torus, maximal operators, weights, and bases in Section~\ref{sec:torus-maximal-weights}. In Section \ref{sec:noweight} we provide the proofs of Theorems \ref{unweighted_unboundedness} and \ref{thm:r_0} concerning the unweighted setting, and remarks on the implication on differentiation properties of bases. Section \ref{sec:pesos} is devoted to the weighted setting: we show several examples of weights and constructions via periodization, and we proof Theorem \ref{thm:unboundedweights_new}. Finally, in Section \ref{sec:reverse} we gather open questions that arise naturally from the present work. 

\section{Preliminaries}\label{sec:torus-maximal-weights}

We introduce the basic definitions for the infinite-dimensional torus as well as the fundamentals for the study of weighted and unweighted inequalities for maximal operators defined over bases of subsets of the torus.

\subsection{Infinite-dimensional torus \texorpdfstring{$\tom$}{Tom}}
\label{subsec:tom}

  For a given $n\in\mathbb{N}$ and a point $x\in \tom=\mathbb{T}^n\times\mathbb{T}^{n,\omega}$, where  $\mathbb{T}^{n,\omega}$ is just a copy of the infinite-dimensional torus $\tom$, we will denote $x_{(n}:=\pi_{\mathbb{T}^n}(x)=(x_1,x_2,\ldots,x_n)$ and $x^{(n}:=\pi_{\mathbb{T}^{n,\omega}}(x)=(x_{n+1},x_{n+2},\ldots)$, the canonical projections of $x$ over $\mathbb{T}^n$ and $\mathbb{T}^{n,\omega}$, respectively. With a slight abuse of notation, we will use $\pi_k$ to refer to the projection over the $k$-th component. We will always denote by $0$ the identity of the group and it will be clear from the context whether we are talking about the identity of the finite-dimensional torus or not. 
  
  We will consider in $\tom$ the product topology, which gives $\tom$ structure of metrizable space. Indeed, the distance function 
\begin{equation}
\label{Saks}
d(x,y):=\sum_{n=1}^\infty \frac{|x_n-y_n|}{2^n},\qquad x,y\in\tom,
\end{equation}
defines a metric in $\tom$ \cite[p. 157]{saks} which is compatible with its topology. We write $\delta(S):=\sup_{x,y\in S}d(x,y)$ for the diameter of the set $S\subset \tom$. 

The $\sigma$-algebra of Borel subsets of $\tom$ is generated by the intervals $I=\prod_{j=1}^\infty I_j$, where $I_j$ is an interval in $\mathbb{T}$ for every $j\in\mathbb{N}$ and there is some $N>0$ such that $I_j=\mathbb{T}$ for every $j> N$. The normalized Haar measure $d x$ on $\tom$ coincides with the product of countable infinite many copies of Lebesgue measure on $\mathbb{T}$ and for a Borel set $E\subset \tom$ we will denote its measure by
\[
|E|:=\int_{\tom}\chi_E(x)\,dx,
\] 
where $\chi_E$ is the characteristic function of the set $E$.
 It was shown in \cite[Chapter~2.3]{TesisEmilio} that $\tom$ with the Haar measure $dx$ and the metric $d$ is not a space of homogeneous type in the sense of Coifman and Weiss \cite{cowe}, and we do not know whether we can equip the space with a metric to make it of homogeneous type (see \cite[Note 2.34]{TesisEmilio}).

In \cite{JLR78}, Jos\'e L. Rubio de Francia showed a result on differentiation of integrals in the context of a locally compact group $G$, that contained a decomposition of Calder\'on--Zygmund type under certain conditions. Later, the third and fourth authors proved in \cite[Theorem~1.2]{Paternostro2019} that a class of $A_\infty$ weights in a locally compact abelian group $G$ admitting covering families satisfies a sharp weak reverse H\"older inequality. The covering families are then understood as a class of base sets which play the same role as cubes in the Euclidean setting, and they are a particular case of what is called a $D'$-sequence, see \cite[Definition~44.10]{Hewitt1970}. By using these sets, the authors were able to prove a Calder\'on--Zygmund type decomposition.  Then they used the latter to prove the sharp reverse H\"older inequality for $A_\infty$ weights and, as an application, they obtained an abstract counterpart of the celebrated Buckley's sharp estimate for the weighted boundedness of the Hardy--Littlewood maximal operator. 

As observed in \cite[p.~194]{Edwards1965}, there are no $D'$-sequences in $\tom$, the infinite-dimensional torus. {Hence, we lack the existence of the base sets considered in \cite{Paternostro2019} and in particular, we lack the automatic validity of the Lebesgue differentiation theorem. 
Even in case we are able to find a different base of sets to work with, we may still lack the corresponding version of Lebesgue differentiation theorem. The validity of such a differentiation property will depend on the particular choice of the basis.}

\subsection{Maximal operator and weights}
\label{subsec:maxw}

A basis $\mathcal{B}$ will be a collection of Borel measurable subsets of the infinite-dimensional torus with finite positive measure. Associated with such a basis, there is a maximal function $M^\mathcal{B}$ which is defined as
\begin{equation}\label{maximal}
M^\mathcal{B}f(x):=\sup_{\mathcal{B}\ni Q\ni x} |f|_Q,   
\end{equation}
where $f_E:=\frac{1}{|E|}\int_Ef\,dx$ is the average value of $f$ on $E$.

Standard harmonic analysis (e.g., harmonic analysis in the Euclidean setting) is strongly based on the nice behavior of the Hardy--Littlewood maximal operator, which  is the maximal operator associated with the basis of intervals. This is because the boundedness properties of a large class of operators are dominated by those of the maximal operator. 

For $q \in [1,\infty)$ let $L^q(\tom,dw)$ be the space of measurable functions $f$ such that
\[
\|f\|_{L^q(\tom,dw)}:=\Big(\int_{\tom}|f(x)|^qw(x)\,dx\Big)^{1/q}<\infty.
\]
Here $w$ is any weight in $\tom$, i.e., any nonnegative integrable function in the infinite-dimensional torus. Similarly, we define $L^\infty(\tom,dw)$, referring to the quantity
\[
\|f\|_{L^\infty(\tom,dw)}:= \inf \{C>0:\, w(\{x \in \tom: |f(x)|>C\})=0\}.
\]
If $w \equiv 1$, then we just omit the measure in the notation. 

{
As usual, for a given sublinear operator $T$ we say that:
\begin{enumerate}
    \item it is of weighted strong type $(q,q)$, $q \in [1, \infty]$, with respect to $w$ (or it is bounded on $L^q(\tom,dw)$) if there exists $C>0$ such that 
    \begin{equation}\label{eq:strong_maximal}
\| Tf\|_{L^q(\tom,dw)}\leq C \| f\|_{L^q(\tom,dw)},\quad f\in L^q(\tom,dw),
 \end{equation}
 \item it is of weak type $(q,q)$, $q \in [1, \infty)$, with respect to $w$ (or it is  bounded from $L^q(\tom,dw)$ to $L^{q, \infty}(\tom,dw)$) if  there exists $C>0$ such that 
 \begin{equation}\label{eq:weak_maximal}
\| f\|_{L^{q,\infty}(\tom,dw)} \leq C \| f\|_{L^q(\tom,dw)},\quad f\in L^q(\tom,dw),
\end{equation}
where
\[
\| f\|_{L^{q,\infty}(\tom,dw)} := \sup_{t>0} t\, w(\{ x\in \tom: |f(x)|>t \})^{1/q},
\]
\item and it is of restricted weak type $(q,q)$, $q \in [1, \infty)$, with respect to $w$ if \eqref{eq:weak_maximal} holds for characteristic functions of measurable sets.
\end{enumerate}
 It follows directly from the above definitions that, for fixed $T$ and $q \in [1, \infty)$, being of strong type $(q,q)$ implies being of weak type $(q,q)$, which in turn implies being of restricted weak type $(q,q)$. For $q=1$ the latter implication can be reversed. }

Two important classes of weights in harmonic analysis are the classes of Muckenhoupt and reverse H\"older weights. These are defined relying on a particular basis $\mathcal{B}$, and they are closely connected with the corresponding maximal operator $M^\mathcal{B}$, see \cite{Duoandikoetxea2016}.

\begin{definition}\label{def:weights1} Let $w\in L^1(\tom)$ be a weight and consider a basis $\mathcal{B}$ in $\tom$.
\begin{enumerate}
\item Let $1<p<\infty$. We say that $w$ is in $A_{p}^\mathcal{B}(\tom)$ if
\begin{equation}\label{Ap}
[w]_{A_{p}^\mathcal{B}(\tom)}:= \sup_{Q\in\mathcal{B}}\left(\frac{1}{|Q|}\int_Q w(x)dx\right)\left(\frac{1}{|Q|}\int_Q w(x)^{1-p'}dx\right)^{p-1}<\infty,
\end{equation}
where $p'$ is the dual exponent of $p$, defined by $1/p+1/p'=1$. Weights satisfying this condition for some $p \in (1, \infty)$ are called Muckenhoupt weights with respect to $\mathcal{B}$.
\item Let $r \in (1, \infty)$. We say that $w\in \mathrm{RH}_r^\mathcal{B}(\tom)$ if there is some finite constant $C>0$ such that the reverse H\"older inequality
\begin{equation}\label{RHr}
\left(\frac{1}{|Q|}\int_Q w(x)^r\,dx\right)^{1/r}\leq C \frac{1}{|Q|}\int_Q w(x) \,dx
\end{equation}
is satisfied for every base set $Q\in \mathcal{B}$. The smallest constant in the above inequality will be denoted by $[w]_{\mathrm{RH}_r^\mathcal{B}(\tom)}$. Weights satisfying this condition for some $r>1$ are called reverse H\"older weights with respect to $\mathcal{B}$.
\end{enumerate}
 \end{definition}

\begin{remark}
These definitions are equivalent in the Euclidean setting if the basis is the family of all cubes in $\mathbb{R}^n$ and they all define what is called the class of $A_\infty(\mathbb{R}^n)$ weights (in other words, $\bigcup_{p \in (1, \infty)}A_p(\mathbb{R}^n) = \bigcup_{r \in (1, \infty)} \mathrm{RH}_r(\mathbb{R}^n) = A_\infty(\mathbb{R}^n)$).  
\end{remark}

For $p=1$, we have the class of $A_1^\mathcal{B}(\tom)$ weights, which we define in the usual way, by the uniform boundedness of the ratio between the maximal function and the original function, i.e., a weight $w\in L^1(\tom)$ is an $A_1^\mathcal{B}(\tom)$ weight if 
 \begin{equation}
 \label{eq:A1}
 [w]_{A_1^\mathcal{B}(\tom)}:= \Big\| \frac{M^\mathcal{B}w}{w
} \Big\|_{L^\infty(\tom)}<\infty.
 \end{equation}

It turns out that in the general case there is no {\it a priori} relation between the classes of weights we have defined, see Section~\ref{sec:reverse}.
Also, in general, nothing guarantees that the maximal operator associated with a given basis $\mathcal{B}$ is bounded neither in the unweighted case nor in the weighted case. The results in this paper, as announced in the beginning, are developed in this direction.  

The starting point for the topics studied in the present paper can be found in the works by Jessen \cite{Jessen1,Jessen2}, where it is proved that the apparently natural basis $\mathcal{J}$ of all the intervals of the infinite-dimensional torus does not differentiate $L^1(\tom)$ functions and therefore the associated maximal operator cannot be of weak type $(1,1)$. This shows that some difficulties must arise when one tries to do harmonic analysis in the infinite-dimensional torus.

\subsection{Rubio de Francia bases \texorpdfstring{$\mathcal{R}_0$}{R0} and \texorpdfstring{$\mathcal{R}$}{R}}\label{sub:bases}
 
In this subsection we introduce the precise definitions of the bases $\mathcal{R}_0$ and $\mathcal{R}$ on which we base our study. The \textit{Rubio de Francia basis} $\mathcal{R}$, originally devised by Jos\'e Luis Rubio de Francia, was introduced in \cite[Section~2.2]{Fernandez-Roncal2020}. We will call their elements cubes, since they  will play the same role as the cubes in the Euclidean setting. Although the maximal function related to this basis $\mathcal{R}$, as recently showed by the first author \cite{Kosz2020}, is not of weak type $(1,1)$,  we can find inside this class $\mathcal{R}$ a subclass $\mathcal{R}_0$ (the \textit{restricted Rubio de Francia basis}) of base sets which play the role of dyadic cubes, and so we will give the name of dyadic cubes to its elements. More precisely,  we will consider the family of dyadic intervals in $\mathbb{T}^\omega$ introduced by Rubio de Francia and defined in \cite{Fernandez-Roncal2020}, from which we will borrow the notation (see also the Ph.D. dissertation \cite{TesisEmilio}). 

  For every $k\in \mathbb{N}$, consider the set $R_k:= \{0,1/k,\ldots, (k-1)/k\}$ of all $k$-th roots of the unity in $\mathbb{T}$. Take the subgroups $H_0=\{0\}$ and $H_1:=R_2\times \{0^{(1}\}$ and, for every $n\geq 1$, 
 \begin{equation}
 \label{eq:Hs}
 H_{n^2+j} := \begin{cases} 
 R_{2^n}^n\times R_{2^j}\times\{0^{(n+1}\},& \text{if }1\leq j\leq n,\\
 R_{2^{n+1}}^{j-n}\times R_{2^n}^{2n+1-j}\times\{0^{(n+1}\},& \text{if }n+1\leq j\leq 2n+1.
 \end{cases}
 \end{equation}
 This defines an increasing sequence $H_1\subset H_2\subset \cdots \subset H_m\subset\cdots$ of subgroups of $\mathbb{T}^\omega$ satisfying $[H_{m+1}:H_m]=2$ for every $m$ and $\mathbb{T}^\omega=\overline{\bigcup_{m\in \mathbb{N}}H_m}$, see \cite{Fernandez-Roncal2020} and \cite[p.~41]{TesisEmilio}. Actually, after $H_1$ we define, for each $n\ge 1$,
 $$
 H_{n^2+j}=\widetilde{H}_{n^2+j}\times\{0^{(n+1}\} \quad (1\le j\le 2n+1),
$$
 and
 $$
 \widetilde{H}_{n^2+j}
:=\begin{cases} 
 R_{2^n}^n\times R_{2^j},& \text{if }1\leq j\leq n,\\
 R_{2^{n+1}}^{j-n}\times R_{2^n}^{2n+1-j},& \text{if }n+1\leq j\leq 2n+1.
 \end{cases}
 $$
 
 We can then define a decreasing sequence $\{V_m\}_{m\in \mathbb{N}}$ of open sets in $\mathbb{T}^\omega$ as follows. Define the open sets $V_0=\tom$ and $V_1:=\left(0,\frac{1}{2}\right)\times \mathbb{T}^{1,\omega}$ and, for every $n\geq 1$,
 \begin{equation}
 \label{eq:fundamental}
 V_{n^2 +j}:= \begin{cases} 
\left(0,\frac{1}{2^n}\right)^n\times \left(0,\frac{1}{2^j}\right)\times \mathbb{T}^{n+1,\omega},&\text{if }1\leq j\leq n,\\
\left(0,\frac{1}{2^{n+1}}\right)^{j-n}\times \left(0,\frac{1}{2^n}\right)^{2n+1-j}\times \mathbb{T}^{n+1,\omega},&\text{if }n+1\leq j\leq 2n+1.
 \end{cases}
 \end{equation}
 Note that for a given $m\in \mathbb{N}$ there are $n\in \mathbb{N} \cup \{0\}$ and $j\in\{1,\ldots,2n+1\}$ such that $m=n^2+j$. The set $V_{m}$ has measure $2^{-m}=2^{-(n^2+j)}$ and, for a given $x\in V_m$, the   last coordinates $x^{(n+1}$ of $x$ vary freely in $\mathbb{T}^{n+1,\omega}$. Given any positive integer $m$ written in this way, we will denote $\ell(m)=n+1$, the number of \textit{nonfree components} of the    set $V_m$, defining also $\ell(0)=0$. We can also write, for each $n\ge 1$,
\begin{equation}
\label{uvesubn}
V_{n^2+j}=\widetilde{V}_{n^2+j}\times\mathbb{T}^{n+1,\omega}\quad (1\le j\le 2n+1).
\end{equation}
where $\widetilde{V}_{n^2+j}$ is the nonfree part of $V_{n^2+j}$, that is 
\begin{equation*}
\widetilde{V}_{n^2+j}:= \begin{cases} 
\left(0,\frac{1}{2^n}\right)^n\times \left(0,\frac{1}{2^j}\right),&\text{if }1\leq j\leq n,\\
\left(0,\frac{1}{2^{n+1}}\right)^{j-n}\times \left(0,\frac{1}{2^n}\right)^{2n+1-j},&\text{if }n+1\leq j\leq 2n+1.  
 \end{cases}
\end{equation*}
See Figure~\ref{CZ_fig1}\footnote{Pictures taken from \cite{TesisEmilio}. We thank Emilio Fern\'andez for allowing us to include them here.} for the projections of the first nine sets $V_m$.

\begin{figure}[ht] 
\begin{center}
\includegraphics[scale=.65]{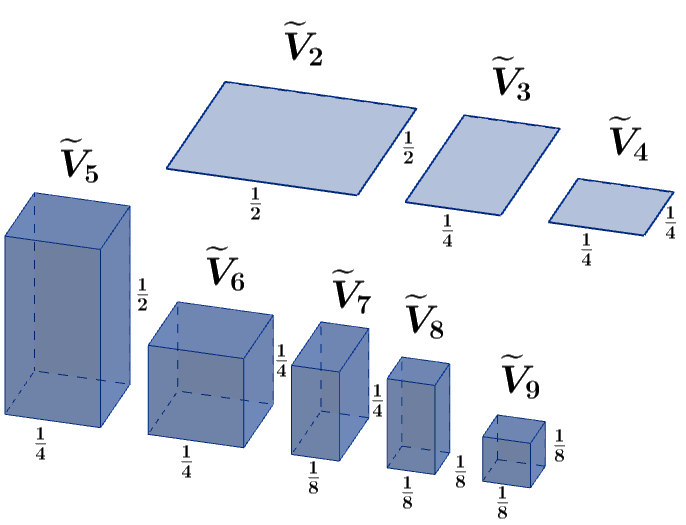}
\caption{First members of the sequence $\{\widetilde{V}_m\}_{m\in\mathbb N}$, e.g., $\widetilde{V}_7=(0,\frac 18)\times(0,\frac 14)^2$.
Two translations of $V_{m+1}$ by elements of $H_{m+1}$ cover (a.e.) $V_m$. }\label{CZ_fig1}
\end{center}
\end{figure}

Let $\mathcal{N}_m := \{t+V_m:t\in H_m\}$, $m\geq0$.  {Since for each $m\in\mathbb{N}$ the set $V_m$ is a fundamental domain for the quotient group $\mathbb{T}^\omega/H_m$,}  {it turns out that $\mathcal{N}_m$ is a finite family of disjoint open subsets of $\mathbb{T}^\omega$ whose union is $\mathbb{T}^\omega$ except for a zero measure set.} This sequence satisfies that, given $m\geq1$, for every $Q\in \mathcal{N}_m$ there is a unique  $Q'\in \mathcal{N}_{m-1}$  {with $Q\subset Q'$. Therefore,} for a.e. $x\in \mathbb{T}^\omega$ there is a unique decreasing sequence $\{Q^{(m)}(x)\}_{m\geq 0}$ with $Q^{(m)}(x)\in \mathcal{N}_m$ and $x\in Q^{(m)}(x)$ for every $m\geq 0$. Observe that $m=-\log_2|Q|$ for every $Q\in\mathcal{N}_m$. We will call $m$ the \textit{sizelevel} of the dyadic cubes in $\mathcal{N}_m$.

The restricted Rubio de Francia basis is the family
\[
\mathcal{R}_0:=\bigcup_{m\geq 0}\mathcal{N}_m.
\]
The basis $\mathcal{R}_0$ is a dyadic decomposition of $\tom$ and, as mentioned above, its elements will be called dyadic cubes. Similarly, we will say that the elements of $\bigcup_{m\in\mathbb{N}}H_m$ are dyadic translations. The   Rubio de Francia basis is the family
\[
\mathcal{R}:=\bigcup_{g\in \tom}\bigcup_{m\geq0}\bigcup_{Q\in\mathcal{N}_m}g+Q,
\]
and the elements of this family   will be called just cubes.   The concepts of nonfree components and sizelevel of a dyadic cube are translation invariant and they extend to nondyadic cubes.  {For a cube $Q\in\mathcal{R}$, we denote by $m_Q$ its sizelevel and, consequently, by $\ell(m_Q)$ the number of its nonfree components.}

\section{Unweighted setting}\label{sec:noweight}

As mentioned in Section~\ref{sec:Intro}, it was proved in \cite{Fernandez-Roncal2020} that the dyadic   maximal operator $ M^{\mathcal{R}_0}$ is of weak type $(1,1)$. Thus, by interpolating with the trivial $L^\infty$ bound, we get that it is of strong type $(q,q)$ for any $q \in (1, \infty)$. On the other hand, the first author showed that $M^{\mathcal{R}}$ is not of weak type $(1,1)$ (see \cite[Proposition~3.1]{Kosz2020}). 

A natural question in this regard is to determine the precise threshold basis $\mathcal{R}'$, with $\mathcal{R}_0\subset\mathcal{R}'\subset\mathcal{R}$, beyond which the behavior of the maximal operator $M^{\mathcal{R}'}$ changes from being of restricted weak type $(q,q)$ for $q \in [1, \infty)$ to not being it.  In other words, we are interested in determining  the smallest basis $\mathcal{R}_0\subset\mathcal{R}'\subset\mathcal{R}$ for which the associated maximal operator \[M^\mathcal{R'}f(x):=\sup_{\mathcal{R'}\ni Q\ni x}|f|_Q \]
is not of restricted weak type $(q,q)$ for a given $q \in [1, \infty)$.

We have not been able to give a fully satisfactory answer to this question. Instead, we have found a way to build bases  $\mathcal{R}_0\subset \mathcal{R}'\subset \mathcal{R}$ with prescribed boundedness properties for the related maximal operator $M^{\mathcal{R}'}$.   In particular we will deduce, as a corollary, the unboundedness of the maximal function $M^{\mathcal{R}}$ for all $q \in [1, \infty)$. 

Note that any basis   $\mathcal{R}_0\subset\mathcal{R}'\subset\mathcal{R}$ is of the form $\mathcal{R}'=\mathcal{R}_0\cup S$, where $S$ is some subset of $\mathcal{R}$, that is, a set of translated dyadic cubes.  In principle, although the collection of cubes of $S$ may not be countable, it happens that the set of sizes of cubes of $S$ has to be countable, by the structure of the Rubio de Francia basis itself. We remark that for the case $\sharp(\{|Q|\}_{Q\in S})<\infty$ (where $\sharp(\cdot)$ is counting measure on $\mathbb{R}$) we have boundedness, as the following lemma shows.

 \begin{lemma}\label{lematamanos}
 Let $S\subset \mathcal{R}$. If $\sharp(\{|Q|\}_{Q\in S})<\infty$, then $M^{\mathcal{R}_0\cup S}$ is of weak type $(q,q)$ for any $q \in [ 1, \infty)$.
 \end{lemma}
 \begin{proof}
Since $M^{\mathcal{R}_0}$ is of weak type $(1,1)$, it follows from interpolation that it is also a bounded operator on $L^q(\tom)$, $q \in (1, \infty)$. Let $q \in [1, \infty)$. Then, given any $x\in \tom$ and any $f\in L^q(\tom)$,
\[
M^{\mathcal{R}_0\cup S}f(x)\leq \max \Big\{M^{\mathcal{R}_0}f(x), \, \sup_{Q\in S} |f|_Q \Big\},
\]
and by Jensen's inequality we have that
\[
|f|_Q = \frac{1}{|Q|}\int_{Q}|f(y)|dy\leq \Big(\frac{1}{|Q|}\int_{Q}|f(y)|^qdy\Big)^{1/q}\leq \frac{\|f\|_{L^q(\tom)}}{|Q|^{1/q}}.
\]
Therefore,  if  $Q_1,\ldots,Q_N\in S$ are such that $\{|Q|\}_{Q\in S}=\{|Q_j|\}_{j=1}^N$, then
\begin{align*}
\|M^{\mathcal{R}_0\cup S}f\|_{L^{q,\infty}(\tom)}&\leq C \max \Big\{ \| M^{\mathcal{R}_0}f \|_{L^{q,\infty}(\tom)}, \frac{\|f\|_{L^q(\tom)}}{|Q_1|^{1/q}},\ldots, \frac{\|f\|_{L^q(\tom)}}{|Q_N|^{1/q}}
 \Big\}\\
 &\leq C \max\Big\{ \|M^{\mathcal{R}_0}\|_{L^{q}(\tom)\to L^{q,\infty}(\tom)},\max_{j=1,\ldots, N}\frac{1}{|Q_j|^{1/q}}\Big\}\|f\|_{L^q(\tom)},
 \end{align*}
  which proves that $M^{\mathcal{R}_0\cup S}$ is of weak type $(q,q)$.
 \end{proof}
 
The previous lemma says that it is enough to study an extension of  $\mathcal{R}_0$ with cubes of countably many different sizes since otherwise the corresponding maximal operator is of weak type $(q,q)$. As we will see in the following lemma, apart from the latter, it is also enough to work just with an extension $S$ of $\mathcal{R}_0$ which is countable itself.

\begin{lemma}\label{lem:countable}
Let $S\subset \mathcal{R}$ be an infinite set. Then there exists a countable subset $S' \subset S$ such that $M^{\mathcal{R}_0\cup S'}f = M^{\mathcal{R}_0\cup S}f$ holds for each $f \in L^1(\tom)$. In particular, for given $q \in [1, \infty)$, the operator $M^{\mathcal{R}_0\cup S}$ is of restricted weak type $(q,q)$ if and only if $M^{\mathcal{R}_0\cup S'}$ is.
\end{lemma}

\begin{proof}
The proof relies on the fact that for each $n \in \mathbb{N}$ the space $\mathbb{T}^n$ with its natural topology is separable.

For each $m \in \mathbb{N}$ consider the set $S_m \subset S$ of all elements of $S$ which are translations of $V_m \in \mathcal{R}$ (we allow the case $S_m = \emptyset$). Write $m = n^2 + j$, where $n \in \mathbb{N} \cup \{0\}$ and $j \in \{ 1, \dots, 2n+1\}$. Then each $Q \in S_m$ is of the form $Q = T_Q+V_m$ for some $T_Q = (t_1, \dots, t_{n+1}, 0, \dots) \in \tom$. Denote
\[
{\bf Q}^{n+1}:= \left[ \mathbb{Q}\cap[0,1)\right]^{n+1}\times \{0^{(n+1}\}
\]
and, for each ${\bf q}\in {\bf Q}^{n+1}$ and each $l \in \mathbb{N}$, define  $S'(m, {\bf q}, l)$ to be  {the singleton of} one of the cubes $Q\in S_m$ with $d({\bf q},T_Q)<2^{-l}$ if there is any. Otherwise set $S'(m, {\bf q}, l)=\emptyset$.

Finally, take
\[
S' :=  \bigcup_{m\in \mathbb{N}}\bigcup_{{\bf q}\in  {\bf Q}^{\ell(m)}}\bigcup_{l\in\mathbb{N}} S'(m, {\bf q}, l).
\]
Of course, by construction, $S'$ is countable. Moreover, if $Q \in S \setminus S'$, then by the density of rationals there exists a sequence $\{Q_j\}_j \subset S'$ approaching $Q$ and, by the Lebesgue dominated convergence theorem, for a given function $f \in L^1(\tom)$ we have $\lim_{j \rightarrow \infty} |f|_{Q_j} = |f|_Q$. 

The inequality $M^{\mathcal{R}_0\cup S'}f\leq M^{\mathcal{R}_0\cup S}f$ is obvious. For finishing the proof note that, for each $x\in\tom$ and each $\varepsilon>0$, there is a cube $Q\in S$ containing $x$ with $  M^{\mathcal{R}_0\cup S}f(x)\leq|f|_Q+\varepsilon/2$.    By the proved convergence, there is, associated with $Q$, a cube $\widetilde{Q}\in S'$ containing $x$ with $||f|_Q-|f|_{\tilde{Q}}|<\varepsilon/2$. Therefore, \[
\begin{split}
M^{\mathcal{R}_0\cup S}f(x) &\leq |f|_Q+\frac{\varepsilon}2 = |f|_Q-|f|_{\tilde{Q}}+|f|_{\tilde{Q}}+\frac{\varepsilon}2 \leq |f|_{\tilde{Q}}+\varepsilon\leq 
M^{\mathcal{R}_0\cup S'}f(x)+\varepsilon.
\end{split}
\]
Since this is valid for any $\varepsilon>0$, we have that $M^{\mathcal{R}_0\cup S}f(x)\leq M^{\mathcal{R}_0\cup S'}f(x)$.
\end{proof}

%%%%%%%%%%%%%%%%%%%%%%%%%%
\subsection{Threshold bases and \texorpdfstring{$(\varepsilon, l)$}{(e,l)}-configurations. Proof of Theorem~\ref{unweighted_unboundedness}}
%%%%%%%%%%%%%%%%%%%%%%%%%%

In view of Lemma~\ref{lem:countable}, in order to seek the threshold basis for the restricted weak $L^q$ unboundedness of the associated maximal operator, we have to find  an appropriate infinite countable set $S=\{T_j+Q_j\}_{j\in\mathbb{N}}$ of translated dyadic cubes $Q_j$, where $\{T_j\}_{j\in \mathbb{N}}$ is a sequence of translations in $\tom$, such that the associated maximal operator $M^{\mathcal{R}'}$, with $\mathcal{R}'=\mathcal{R}_0\cup S$, is not of restricted weak type $(q,q)$.

Notice that $S$ must be chosen carefully since in general it is easy to produce an infinite set $S$ as above such that $M^{\mathcal{R}'}$ is smaller than, say, $2M^{\mathcal{R}_0}$, and hence bounded. Indeed, for each $j$ take arbitrary $R_j \in \mathcal{R}_0$ with its children $R_j'$ and $R_j^{''}$, and let $T_j+Q_j$ be the translate of $R_j'$, lying in the middle of $R_j$ (see Figure~\ref{fig:bdd_construction}). For any nonnegative function, its average over $T_j+Q_j$ is controlled by $2$ times the average over $R_j$, hence we are done. 
\begin{figure}[h]
    \centering
    \begin{minipage}[b]{0.45\textwidth}
    \includegraphics[scale=0.4]{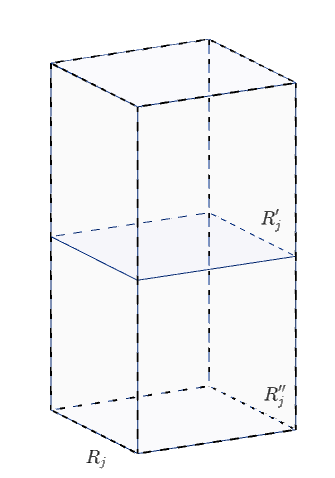}
  \end{minipage}
\
  \begin{minipage}[b]{0.45\textwidth}
      \includegraphics[scale=0.4]{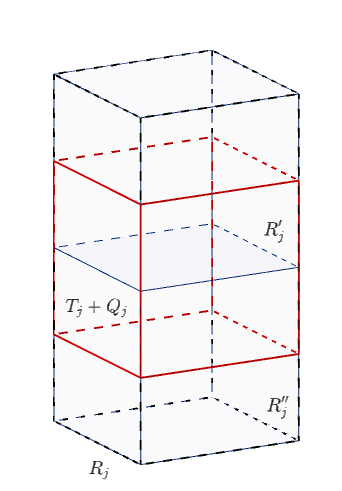}
  \end{minipage}
    \caption{The cube $R_j$ with its two children, $R_j'$ and $R_j''$, at left. The same cubes but now including  the new cube $T_j+Q_j$, at right.}
    \label{fig:bdd_construction}
\end{figure}

Another interesting situation arises if all the elements of $S=\{T_j+Q_j\}_{j\in\mathbb{N}}$ are disjoint. In this case, the considered boundedness properties of $M^{\mathcal{R}_0}$ and $M^{\mathcal{R}'}$ are  the same. Indeed, if $x \in \tom \setminus \bigcup_{j\in\mathbb{N}} T_j+Q_j$, then $M^{\mathcal{R}'}f(x) - M^{\mathcal{R}_0}f(x) = 0$, while  in case $x \in T_j+Q_j$ for some $j\in\mathbb{N}$ we have $0 \leq M^{\mathcal{R}'}f(x) - M^{\mathcal{R}_0}f(x) \leq |f|_{T_j+Q_j}$. Thus, for each $q \in [1, \infty]$ we have \[
\|M^{\mathcal{R}'}f - M^{\mathcal{R}_0}f \|_{L^q(\tom)} \leq \big\|  \sum_{j\in\mathbb{N}} \chi_{T_j+Q_j} \, |f|_{T_j+Q_j}  \big\|_{L^q(\tom)} \leq \|f\|_{L^q(\tom)}.\]

Even more, when the elements of $S$ have bounded overlapping, $M^{\mathcal{R}_0}$ and $M^{\mathcal{R}'}$ have also the same boundedness properties. 

\begin{lemma}
\label{lem:bddover}
Let $\mathcal{R}'=\mathcal{R}_0\cup S$, with $S=\{T_j+Q_j\}_{j\in \mathbb{N}}$. If 
$$
\big\|  \sum_{Q\in S}\chi_{Q} \big\|_{L^\infty(\tom)} \le N
$$ 
for some $N\in \mathbb{N}$, then for $q \in [1, \infty]$ we have
$$
\|M^{\mathcal{R}'}f - M^{\mathcal{R}_0}f\|_{L^q(\tom)}\le N \|f\|_{L^q(\tom)}.
$$
\end{lemma}
\begin{proof}
For each $x\in \tom$, we define $S_x=\{Q\in \mathcal{R}'\setminus\mathcal{R}_0: x\in Q\}$. Then
$$
|M^{\mathcal{R}'}f(x) - M^{\mathcal{R}_0}f(x)|\le \sup_{Q\in S_x}|f|_Q \leq \sum_{Q\in S_x}|f|_Q = \sum_{Q\in S}|f|_Q \chi_Q(x).
$$
Observe that if $S_x=\emptyset$, then the right hand side is zero. Note also that $S=\cup_x S_x$. In the case $q=1$ we have 

\begin{align*}
    \big\| \sum_{Q\in S}|f|_Q \chi_Q \big\|_{L^1(\tom)}
    &=\int_{\tom}\sum_{Q\in S}\frac{1}{|Q|}\int_{\tom}|f(y)|\chi_Q(y)\chi_Q(x)\,dy\,dx\\
    &=\int_{\tom}|f(y)|\sum_{Q\in S}\left(\frac{1}{|Q|}\int_{\tom}\chi_Q(x)\,dx\right)\chi_Q(y)\,dy\\
     &= \int_{\tom}|f(y)|\sum_{Q\in S}\chi_{Q}(y)\,dy\le N\|f\|_{L^1(\tom)}.
\end{align*}
Also, for $q=\infty$ we have $\|\sum_{Q\in S}|f|_Q \chi_Q \|_{L^\infty(\tom)} \leq N \|f\|_{L^\infty(\tom)}$ by a simple pointwise estimate. So, by the Riesz--Thorin interpolation theorem applied to the linear operator $\Sigma^S$ given by 
$
\Sigma^S f(x) := \sum_{Q\in S} f_Q \chi_Q(x),
$
we obtain \[
\|M^{\mathcal{R}'}f - M^{\mathcal{R}_0}f\|_{L^q(\tom)} \leq \big\|\Sigma^S |f| \big\|_{L^q(\tom)} \leq N \big \| |f| \big \|_{L^q(\tom)} = N \| f \|_{L^q(\tom)}\]
for arbitrary $q \in [1, \infty]$.
 \end{proof}
 
As the above discussion shows, one cannot expect that $M^{\mathcal{R}'}$ will always be unbounded. It is hard to find the threshold property of $S$ which makes the family $\mathcal{R}'=\mathcal{R}_0\cup S$ have    an unbounded associated maximal operator $M^{\mathcal{R}'}$. Instead, we have found a condition on $S$ which ensures that the associated maximal operator is large. In case the condition is not satisfied, we can always slightly enlarge $S$ in such a way that the new basis has unbounded associated operator.

In what follows, we will look for a family with unbounded  but still in some sense controlled overlapping. This keeps us close to our primary motivation, that is, finding the accurate threshold condition. In this direction, let us recall
the concept of sparse families of cubes (see \cite[Definition 6.1]{LernerNazarov}). We say that a family $\mathcal{S}$ is sparse if one can find $0<\eta<1$ and a family $\{E_Q\}_{Q\in\mathcal{S}}$ of pairwise disjoint measurable sets satisfying $E_Q\subset Q$ and $\eta|Q|<|E_Q|$ for all $Q\in\mathcal{S}$. For sparse families, we have a control on the overlapping of its elements. Indeed, under the sparse condition, the function $\sum_{Q\in\mathcal{S}}\chi_{Q}$ is at least integrable. 

The structures we use for proving results for maximal operators are the $(\varepsilon, l)$-configurations introduced in Definition~\ref{epsn}.
They involve a more restrictive (but at the same time more precise) condition than just the sparsity of the family. However, the example below indicates that it is still relatively easy to find the said structures among finite families of cubes in $\tom$.  

\begin{example}\label{ex:epsn}
Let $Q \in \mathcal{R}_0$ be such that $\ell(m_{Q}) \geq l$. For each $k \in \{1, \dots, l\}$ set $Q^{(k)} = T^{(k)} + Q$, where $T^{(k)} = (0, \dots, 0, t_k, 0, \dots)$ with $\varepsilon |\pi_k(Q)| \leq t_k \leq (1-\varepsilon)|\pi_k(Q)|$. Then $\{Q^{(1)}, \dots, Q^{(l)}\}$ is an $(\varepsilon, l)$-configuration (with $Q$ playing the role of $A$  {in  Definition \ref{epsn})}. 
\end{example}
Since this example will be of central importance in the case $t_k=(1-\varepsilon) |\pi_k(Q)|$, we will give it a name. 

\begin{definition}
\label{def:confaround}
An $(\varepsilon,l)$-configuration of the type described in Example \ref{ex:epsn} with $t_k=(1-\varepsilon) |\pi_k(Q)|$ will be called an \emph{$(\varepsilon,l)$-configuration around $Q$} (see Figure~\ref{fig:el-conf}).
\end{definition}

\begin{figure}[h]
    \centering
    \begin{minipage}[b]{0.5\textwidth}
    \includegraphics[scale = 1.4]{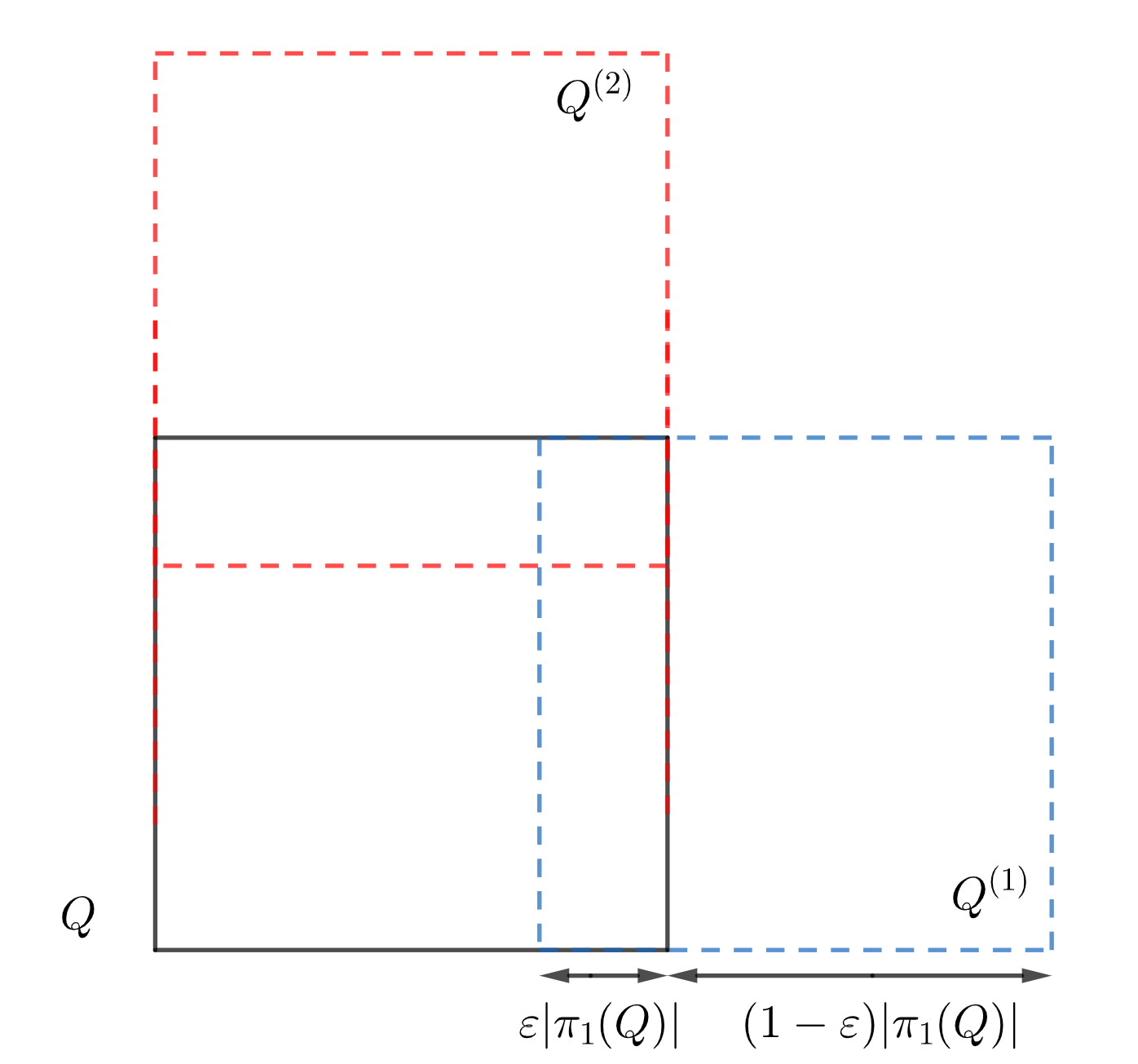} 
  \end{minipage}
  \
  \begin{minipage}[b]{0.45\textwidth}
     \includegraphics[scale=0.25]{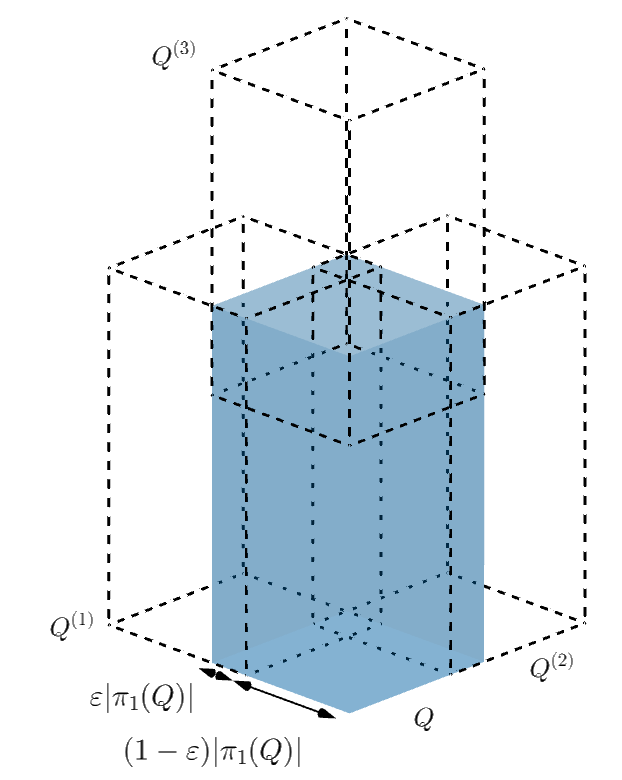}
  \end{minipage}
    \caption{At left, the projection on the first two components of a cube $Q$ with $2$ nonfree components with an $(\varepsilon,2)$-configuration around it. At right, the projection on the first three components of another cube $Q$ with $3$ nonfree components and an $(\varepsilon,3)$-configuration around it.}
    \label{fig:el-conf}
\end{figure}

Observe that each $(\varepsilon,l)$-configuration is a family of $l$ cubes which are  $\varepsilon$-sparse, as we can find a pairwise disjoint family $\{E_{Q^{(k)}}\}_{k=1}^l$ of measurable sets satisfying the definition of sparsity for $\eta=\varepsilon$. Indeed, define for each $k\in \{1,\ldots,l\}$ the set $E_{Q^{(k)}}:=Q^{(k)}\backslash A$. The family  $\{E_{Q^{(k)}}\}_{k=1}^l$ is pairwise disjoint by the definition of $(\varepsilon,l)$-configuration and observe that for each $k\in \{1,\ldots,l\}$ we have $E_{Q^{(k)}}\subset Q^{(k)}$ and $|E_{Q^{(k)}}|=|Q^{(k)}\backslash A| = |Q^{(k)}|-|Q^{(k)}\cap A|\geq \varepsilon |Q^{(k)}|$. 

\begin{remark} It is worth noting that, following Example~\ref{ex:epsn}, one can construct an $(\varepsilon,l)$-configuration such that $Q^{(1)} \cup \dots \cup Q^{(l)}$ is contained in a given cube $\widetilde{Q} \in \mathcal{R}_0$.
\end{remark}

The maximal operators $M^{ \mathcal{R}'}$ associated with families $\mathcal{R}_0\subset  \mathcal{R}'\subset \mathcal{R}$  containing infinitely many appropriately chosen $(\varepsilon,l)$-configurations can be shown to have pathological $L^p$ behavior. We are ready to prove Theorem~\ref{unweighted_unboundedness}.
 
 \begin{proof}[Proof of Theorem~\ref{unweighted_unboundedness}]
  Let $f_j = \chi_{A_j}$, where $A_j$ is the set $A$ from Definition~\ref{epsn} specified for $S_j$. Then for each $x \in Q \in S_j$ we have 
  \[
  M^{\mathcal{R}'} f_j {(x)} \geq |f_j|_{Q} = \frac{|Q \cap A_j|}{|Q|} \geq \varepsilon_j.
  \]
  Thus, since $|Q_j^{(k)} \setminus A_j|  \ge   \varepsilon_j |A_j|$, we obtain
 \begin{equation}
 \label{eps2}
 \frac{\|M^{\mathcal{R}'}f_j\|_{L^{q,\infty}(\tom)}}{\|f_j\|_{L^q(\tom)}} \geq \frac{\varepsilon_j}{2} \, \Bigg(\frac{ \sum_{k=1}^{l_j} |Q_j^{(k)} \setminus A_j|}{|A_j|}\Bigg)^{1/q}
 \geq \frac{1}{2} \, (\varepsilon_j^{q+1} l_j)^{1/q},
 \end{equation} 
  and we are done.  
 \end{proof}

\begin{proof}[Proof of Corollary~\ref{corollary}]
 Observe that, in view of Example~\ref{ex:epsn}, for each $j \in \mathbb{N}$ one can find a $(\frac{1}{2}, j)$-configuration contained in $\mathcal{R}$. Then, since $\lim_{j \rightarrow \infty} \frac{j}{2^{q+1}} = \infty$, an application of Theorem~\ref{unweighted_unboundedness} gives the desired result.  
\end{proof}

So we figured out a structure for a family $S$ and a related condition to make the maximal operator associated with the family $\mathcal{R}'=\mathcal{R}_0\cup S$ an unbounded operator. In view of Lemma~\ref{lem:bddover}, the families we have to consider have unbounded overlapping   $\sum_{Q\in S}\chi_Q$ in the $L^{\infty}(\tom)$ sense. It could be thought that somehow the boundedness properties of the overlapping of the cubes in certain $L^q(\tom)$ sense control the behavior of the associated maximal function $M^{\mathcal{R}'}$. However, this is not the case, as we can find families satisfying the condition on the $(\varepsilon,l)$-configurations with either bounded or unbounded overlapping in the $L^1(\tom)$ sense. Thus, it seems that the behavior of the overlapping $\sum_{Q\in S}\chi_Q$ of the cubes does not  actually govern the boundedness properties of the associated maximal function $M^{\mathcal{R}_0\cup S}$. With this said, we remark that we still do not know what is the precise threshold property governing the boundedness properties of the maximal operators we are studying.

In what follows, we will show that the structure we found is good enough to find bases $\mathcal{R}_0\subset\mathcal{R}'\subset \mathcal{R}$ with prescribed boundedness properties of $M^{\mathcal{R'}}$. Moreover, in the context of certain specific collections of $(\varepsilon,l)$-configurations around cubes, we will introduce conditions which are sharp, i.e., they are sufficient and necessary at the same time.

\subsection{Predetermined range of boundedness. Proof of Theorem~\ref{thm:r_0}}
\label{sub:predetermined-range}
Our next goal is to use the technique introduced earlier in its most effective form. As a result, we solve a certain subtle problem regarding the occurrence of restricted weak type $(q,q)$ inequalities for bases $\mathcal{R}_0 \subset \mathcal{R}' \subset \mathcal{R}$, by using the $(\varepsilon, l)$-configurations with parameters chosen very carefully. Namely, for given $q_0 \in (1, \infty)$ we find $\mathcal{R}'$ such that the restricted weak type $(q,q)$ inequality for the associated maximal operator holds if and only if $q$ belongs to a predetermined range of the form $[q_0,\infty]$ or $(q_0, \infty]$. This result has a very appealing form and follows directly from Theorem~\ref{thm:r_0} proved below. 

\begin{proof}[Proof of Theorem~\ref{thm:r_0}]
First, let us take $q \in (1, \infty)$ and  assume $\sup_{j\in\mathbb{N}}\varepsilon_jl_j^{1/q}=\infty$. Let us see that $M^{\mathcal{R}'}$ is not of restricted weak type $(q,q)$.  Indeed, if $f_j = \chi_{Q_j}$, then for each $x \in Q \in S_j$ we have $M^{\mathcal{R}'} f_j(x) \geq |f_j|_{Q} = \frac{|Q \cap Q_j|}{|Q|} = \varepsilon_j$. Thus, since $|Q_j^{(k)} \setminus Q_j|=|Q_j^{(k)}|-|Q_j^{(k)}\cap Q_j| = (1-\varepsilon_j)|Q_j| \geq |Q_j|/2$, we have that
\begin{equation*}
\frac{\|M^{\mathcal{R}'}f_j\|_{L^{q,\infty}(\tom)}}{\|f_j\|_{L^q(\tom)}} \geq \frac{\varepsilon_j}{2} \, \left(\frac{ \sum_{k=1}^{l_j} |Q_j^{(k)} \setminus Q_j|}{|Q_j|}\right)^{1/q}
\geq \frac{\varepsilon_j l_j^{1/q}}{2^{1+1/q}},
\end{equation*} 
and we are done.

If $q=1$ and $\sup_{j\in\mathbb{N}} l_j = \infty$, then for $U_j = \bigcap_{k=1}^{l_j} Q_j^{(k)}$ we have that $M^{\mathcal{R}'} \chi_{U_j} (x) = \varepsilon_j^{l_j}$ for each $x \in E_j$. Thus, since $|U_j| = \varepsilon_j^{l_j} |Q_j|$ and $|Q_j^{(k)} \setminus Q_j| \geq |Q_j|/2$, we get that
\[
\frac{\|M^{\mathcal{R}'}\chi_{U_j} \|_{L^{1,\infty}(\tom)}}{\|\chi_{U_j}\|_{L^1(\tom)}} \geq \frac{\varepsilon_j^{l_j}}{2} \cdot \frac{ \sum_{k=1}^{l_j} |Q_j^{(k)} \setminus Q_j|}{|U_j|}
= \frac{(1-\varepsilon_j) l_j}{2} \geq \frac{l_j}{4},
\] 
and we finished the first part of the proof.  

Let us now assume that $q=1$ and $\sup_{j\in\mathbb{N}} l_j < \infty$. In this case the elements of $S$ have bounded overlapping. By Lemma~\ref{lem:bddover} the operator $M^S$, and hence also $M^{\mathcal{R}'}$, is of restricted weak type $(1,1)$. 

It remains to prove that if $q \in (1, \infty)$ and $\sup_{j \in \mathbb{N}} \varepsilon_j l_j^{1/q} < \infty$, then $M^{\mathcal{R}'}$ is of restricted weak type $(q,q)$. 
Let us start with a series of reductions:

\begin{enumerate}
    \item \emph{ It suffices to work just with $M^S$}:
    
    Indeed, this follows from the quasi-triangle inequality for $L^{q,\infty}(\tom)$ and the fact that $M^{\mathcal{R}_0}$ is of restricted weak type $(q,q)$ for any $q \in (1, \infty)$.
    \item  \emph{We just have to find a uniform constant $C_q(S)>0$ such that the estimate $\|M^{S_j} \chi_{U} \|_{L^{q,\infty}(\tom)} / |U|^{1/q} \leq C_q(S)$ holds for all $j\in\mathbb{N}$ and all $U\subset E_j$}:
    
    The operator $M^S$ vanishes outside $\bigcup_{j \in \mathbb{N}} E_j$, and so  we can focus only on measurable sets $U \subset \bigcup_{j \in \mathbb{N}} E_j$ such that $|U| > 0$. Let $U_j = U \cap E_j$. Denote $D^\lambda = \{ x \in \tom : M^S \chi_U(x) > \lambda \}$ and $D^\lambda_j = \{ x \in E_j : M^{S_j} \chi_{U_j}(x) > \lambda \}$. Since the sets $E_j$ are disjoint, we have $D^\lambda_j = D^\lambda \cap E_j$. Thus,
\[
\frac{\lambda^q |D^\lambda|}{|U|}
= \frac{\lambda^q (|D^\lambda_1| + |D^\lambda_2| + \dots)}{|U_1| + |U_2| + \dots}
\leq \sup_{j \in \mathbb{N} \, : \, |U_j| > 0} \Big( \frac{\lambda^q |D^\lambda_j|}{|U_j|} \Big) \leq C^q_q(S).
\]
\item \emph{For a given $j\in \mathbb{N}$, we can just focus on the cases $U \subset Q_j$ or $U \subset E_j \setminus Q_j$}:

Indeed, for $U \subset E_j$ we write $U_1 = U \cap Q_j \subset Q_j$ and $U_2 = U \setminus U_1 \subset E_j \setminus Q_j$. Then the claim for $U$ follows from the quasi triangle inequality for $L^{q, \infty}(\tom)$ and the postulated maximal inequality for $U_1$ and $U_2$.
\item \emph{Moreover, we can just focus on the case $U\subset Q_j$ since we have the estimate $\|M^{S_j} \chi_{U} \|_{L^{q,\infty}(\tom)} \leq 2^{1/q}|U|^{1/q}$ for $U \subset E_j \setminus Q_j$}:

Observe that for the case  $U \subset E_j \setminus Q_j$ we can define $U^{(k)} = U \cap Q_j^{(k)}$ for all $k\in\{1,\ldots,l_j\}$. Note now that, for any $\lambda>0$, the level set $\{x\in \tom: M^{S_j}\chi_U(x)>\lambda\}$ can be decomposed as the following disjoint union
\[
\{x\in Q_j: M^{S_j}\chi_U(x)>\lambda\} \cup \bigcup_{k=1}^{l_j} \{ x \in Q_j^{(k)} \setminus Q_j : M^{S_j} \chi_{U^{(k)}}(x) > \lambda \}.
\]
Now, on one hand, we have that
\begin{align*}
   \sup_{\lambda>0}\lambda^q|\{x\in Q_j: M^{S_j}\chi_U(x)>\lambda\}| &\leq  \max_{k \in \{1,\ldots,l_j\}} \left(\frac{|U^{(k)}|}{|Q_j^{(k)}|}\right)^q|Q_j|\\
   &=  \max_{k \in \{1,\ldots,l_j\}}|U^{(k)}| \left(\frac{|U^{(k)}|}{|Q_j^{(k)}|}\right)^{q-1}\\
   &\leq  \max_{k \in \{1,\ldots,l_j\}}|U^{(k)}|\leq |U|,
\end{align*}
where it has been used the fact that $|Q_j^{(k)}|=|Q_j|$ for all $k\in\{1,\ldots,l_j\}$. On the other hand, for each $k\in\{1,\ldots,l_j\}$, 
by a similar argument to the one above we get that
\[
\sup_{\lambda>0} \lambda^q  |\{ x \in Q_j^{(k)} \setminus Q_j : M^{S_j} \chi_{U}(x) > \lambda \}|\leq |U^{(k)}|.
\]
The disjointness of the sets $ Q_j^{(k)} \setminus Q_j$ gives then the inequality
\[
\|M^{S_j} \chi_{U} \|_{L^{q,\infty}(\tom)} \leq 2^{1/q}|U|^{1/q}
\]
in the case $U \subset E_j \setminus Q_j$.

\item \emph{We trivially have that $\|(M^{S_j} \chi_{U}) \cdot \chi_{Q_j} \|_{L^{q,\infty}(\tom)} \leq  |U|^{1/q}$ for any $U\subset Q_j$, and so we just have to study $\|(M^{S_j} \chi_{U}) \cdot \chi_{E_j \setminus Q_j} \|_{L^{q,\infty}(\tom)}$}:

Indeed, as before, for each $x \in E_j$ we have $M^{S_j} \chi_{U}(x) \leq |U| / |Q_j|$, which clearly implies 
\[
\|(M^{S_j} \chi_{U}) \cdot \chi_{Q_j} \|_{L^{q,\infty}(\tom)} \leq \frac{|U|}{|Q_j|} \cdot |Q_j|^{1/q} \leq |U|^{1/q}.
\]

\end{enumerate}

After all these reductions, we are left with the following task: for $q \in (1, \infty)$ and $j\in\mathbb{N}$ we have to find a constant $C_q(S)>0$ such that
\[
\|(M^{S_j} \chi_{U}) \cdot \chi_{E_j \setminus Q_j} \|_{L^{q,\infty}(\tom)}/|U|^{1/q}\leq C_q(S) 
\]
for all measurable $U\subset Q_j$ under the assumption that $\sup_{j\in\mathbb{N}}\varepsilon_j l_j^{1/q}<\infty$. 
Let us call $C_q:=\sup_{j\in\mathbb{N}}\varepsilon_j l_j^{1/q}$. Notice that 
\[
\|M^{S_j} \chi_{Q_j} \|_{L^{q,\infty}(\tom)} / |Q_j|^{1/q} \leq \varepsilon_j l_j^{1/q} \leq C_q
\]
holds uniformly in $j$. Unfortunately, for $U$ being a proper subset of $Q_j$ the situation is much more difficult.

Fix $q \in (1, \infty)$ and $j\in\mathbb{N}$. Let $\lambda>0$ and $U\subset Q_j$. 
By the disjointness of the sets $Q_j^{(k)}\setminus Q_j$, it happens that $M^{S_j}\chi_U(x)$ is constantly equal to $|Q_j^{(k)}\cap U| / |Q_j^{(k)}|$ for all $x\in Q_j^{(k)} \setminus Q_j$. 
Thus, associated with the given parameter $\lambda$, there exist just $m$ cubes $Q_j^{(k)}$ for which $|Q_j^{(k)}\cap U| / |Q_j^{(k)}| >\lambda$ for some $m\in\{1,\ldots,l_j\}$ (we omit the trivial case $m=0$). So, in view of the above, we have that 
\[
|\{x\in E_j\setminus Q_j:M^{S_j}\chi_U(x)>\lambda \} |= (1-\varepsilon_j)m|Q_j|,
\]
where, one more time, we took into account that $|Q_j^{(k)}|=|Q_j|$ for all $k\in \{1,\ldots,l_j\}$. 
By symmetry, we may assume that the $m$ cubes we just mentioned are in fact the $m$ first cubes $Q_j^{(1)},\ldots, Q_j^{(m)}$ of $S_j$. Since we want to study the quantity
\begin{equation}
\label{eq:cantidad}
\lambda|\{x\in E_j\setminus Q_j:M^{S_j}\chi_U(x)>\lambda \} |^{1/q}/|U|^{1/q}
\end{equation}
and we already know the size of the level set of the maximal function, it just remains to study the quantity $\lambda/|U|^{1/q}$.

Define a partition of $Q_j$ as follows
\[
Q_j = Q_{j,0} \cup \dots \cup Q_{j,m},
\]
where each $Q_{j,l}$ consists of the elements $x \in Q_j$ which belong to exactly $l$ cubes from the family $\{Q_j^{(1)}, \dots, Q_j^{(m)}\}$. 
We have 
\[
|Q_{j,l}| =  \binom{m}{l} \varepsilon_j^l (1-\varepsilon_j)^{m-l} |Q_j| 
\]
for each $l\in\{0,\ldots,m\}$ (see Figure~\ref{fig:cubitos}).

\begin{figure}[h]
    \centering
    
    \begin{minipage}[b]{0.45\textwidth}
    \includegraphics[scale=0.3]{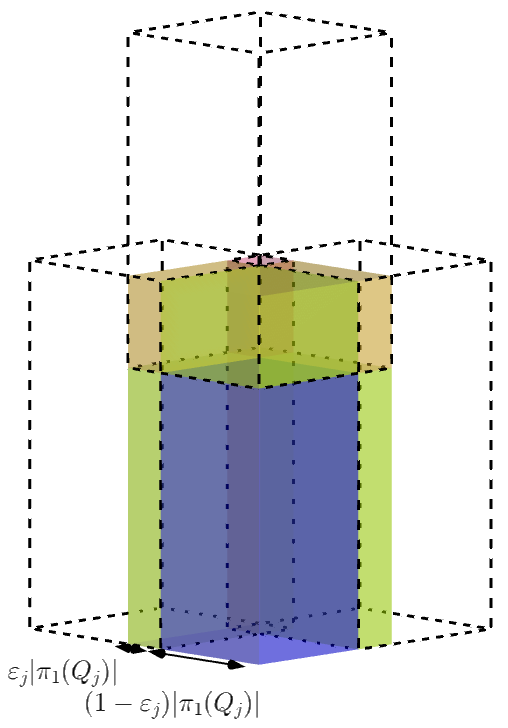}
  
  \end{minipage}
  \
  \begin{minipage}[b]{0.5\textwidth}
     \includegraphics[scale=0.3]{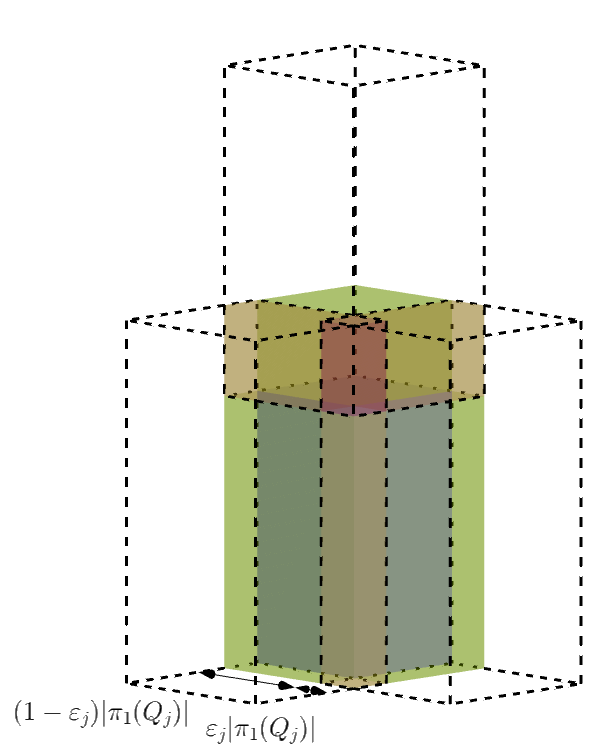}   
  \end{minipage}
    \caption{A generic example (front on the left and back on the right) of a cube $Q_j\in\mathcal{R}_0$ with a $(1/4,3)$-configuration (with dashed lines) wrapping it. In this example, the projection of the cube $Q_j$ to the first three components equals $\widetilde{V}_5$ (the colored region). In blue (at the back) is the set $Q_{j,0}$. The set $Q_{j,1}$ is the green one, $Q_{j,2}$ is the orange one, and $Q_{j,3}$ is the one in red.}
    \label{fig:cubitos}
\end{figure}

 Then, associated with this partition, we consider a random variable $X_{m,j} \colon Q_j\to \{0,1,\ldots, m\}$ defined as $X_{m,j}(w)=l$ if and only if $w\in Q_{j,l}$. This is a binomially distributed random variable with parameters $(m,\varepsilon_j)$, i.e., $X_{m,j}\sim B(m,\varepsilon_j)$, over the underlying probability space $\big(Q_j,\frac{dx}{|Q_j|}\big)$. In other words, the probability  defined by $X_{m,j}$ is  the binomial probability distribution.

Recall that our goal is to estimate \eqref{eq:cantidad} for $\lambda>0$ and $U\subset Q_j$ as above. Since all the set $Q_j^{(k)}$ with $k\in \{1,\ldots,m\}$ satisfy $\frac{|U\cap Q_j^{(k)}|}{|Q_j|}>\lambda$, we have that 
\[
\lambda<m^{-1}\sum_{k=1}^m \frac{|U\cap Q_j^{(k)}|}{|Q_j^{(k)}|}=m^{-1}|Q_j|^{-1}\sum_{k=1}^m \int_{Q_j}\chi_{U\cap Q_j^{(k)}}(x)\, dx,
\]
and, denoting $U_l := U \cap Q_{j,l}$, we get
\[
\begin{split}
\lambda&<m^{-1}|Q_j|^{-1}\sum_{l=0}^{m}\sum_{k=1}^m \int_{Q_j}\chi_{U\cap Q_j^{(k)}}(x)\chi_{Q_{j,l}}(x)\, dx\\
&= m^{-1}|Q_j|^{-1}\sum_{l=0}^{m}\sum_{k=1}^m \int_{Q_j}\chi_{U_l\cap Q_j^{(k)}}(x)\, dx\\
&=m^{-1}|Q_j|^{-1}\sum_{l=0}^{m} l \int_{Q_j}\chi_{U_l}(x)\, dx = |Q_j|^{-1}m^{-1} \sum_{l=0}^m l \cdot |U_l|. 
\end{split}
\]
 Note that the latter quantity is actually the expected value of $X_{m,j}$  over the restricted probability measure $d\mathbb{P}_j$ over $Q_j$, defined as $d\mathbb{P}_j:=\frac{dx}{|Q_j|} $.

In order to get a uniform estimate in $U$, let us just focus on its size. Since $U\subset Q_j$, there is $t\in  (0,1]$ such that $|U|=t|Q_j|$. Then,
\[
\lambda <  m^{-1}|Q_j|^{-1} \sum_{l=0}^m l \cdot |U_l|= m^{-1}\int_U X_{m,j}(x)\, d\mathbb{P}_j(x), 
\]
and we can define 
\begin{equation}\label{eq:Hmeps}
H_{m,\varepsilon_j}(t):=\sup_{U\subset Q_j:|U|=t|Q_j|}\int_U X_{m,j}(x)\, d\mathbb{P}_j(x).
\end{equation}
 Thus, the quantity \eqref{eq:cantidad}  can be bounded by 
\[
\frac{m^{-1} H_{m, \varepsilon_j}(t) \cdot \big( (1-\varepsilon_j) m |Q_j| \big)^{1/q}}{ \big(t |Q_j| \big)^{1/q}} = (1-\varepsilon_j)^{1/q} \, \frac{H_{m, \varepsilon_j}(t)}{t^{1/q} m^{1-1/q}}\le  \frac{H_{m, \varepsilon_j}(t)}{t^{1/q} m^{1-1/q}}, 
\] 
so it remains to show that the above
can be estimated uniformly in $j \in \mathbb{N}$, $m \in \{1, \dots, l_j \}$, and $t \in (0,1]$. 

To remove the dependence on $j$, recall that $\varepsilon_j l_j^{1/q} \leq C_q$ holds uniformly in $j$, which implies $\varepsilon_j \leq C_q / m^{1/q}$, provided that $m \leq l_j$. We let ${X}_m \colon [0,1]\to \{0,1,\ldots,m\}$ be a binomially distributed random variable with parameters $m$ and $\min\{1, C_q / m^{1/q}\}$ and define, by analogy with \eqref{eq:Hmeps},
\[
H_m(t):=\sup_{E\subset [0,1]:|E|=t}\int_E {X}_m(x)\, d\mathbb{P}(x),
\]
where $d\mathbb{P}$ is Lebesgue measure on $[0,1]$. We have that ${X}_m$ is a piecewise function on $[0,1]$, and for the given set  $E$, let us call $E_l={X}_m^{-1}(l)\cap E$. 

Each of the integrals $\int_E {X}_m(x)\, d\mathbb{P}(x)$ is the sum of the areas of all the rectangles  with basis $E_l$ and height $l$ with $l\in \{0,1,\ldots,m\}$. Let us call $X_m^*$ to the decreasing rearrangement of the random variable ${X}_m$, which also follows a binomial distribution $B\left(m,\min\{1,C_q/m^{1/q}\}\right)$, and note that the choice of $E$ that makes the integral $\int_E X_m^*(x)\, d\mathbb{P}(x)$ be the largest among all the sets $E$ with $|E|=t$ is precisely $E=[0,t]$ (see Figure~\ref{fig:rearr}). Then,
\[
H_m(t)= \int_0^t X_m^*(s)\, ds.
\]

Note that the same argument works for the function $H_{m,\varepsilon_j}$, and then we can compute it as
\[
H_{m,\varepsilon_j}(t)=\int_0^t X_{m,j}^*(s)\, ds.
\]

\begin{figure}[h!]
    \centering
    \begin{minipage}[b]{0.4\textwidth}
    \includegraphics[scale=0.6]{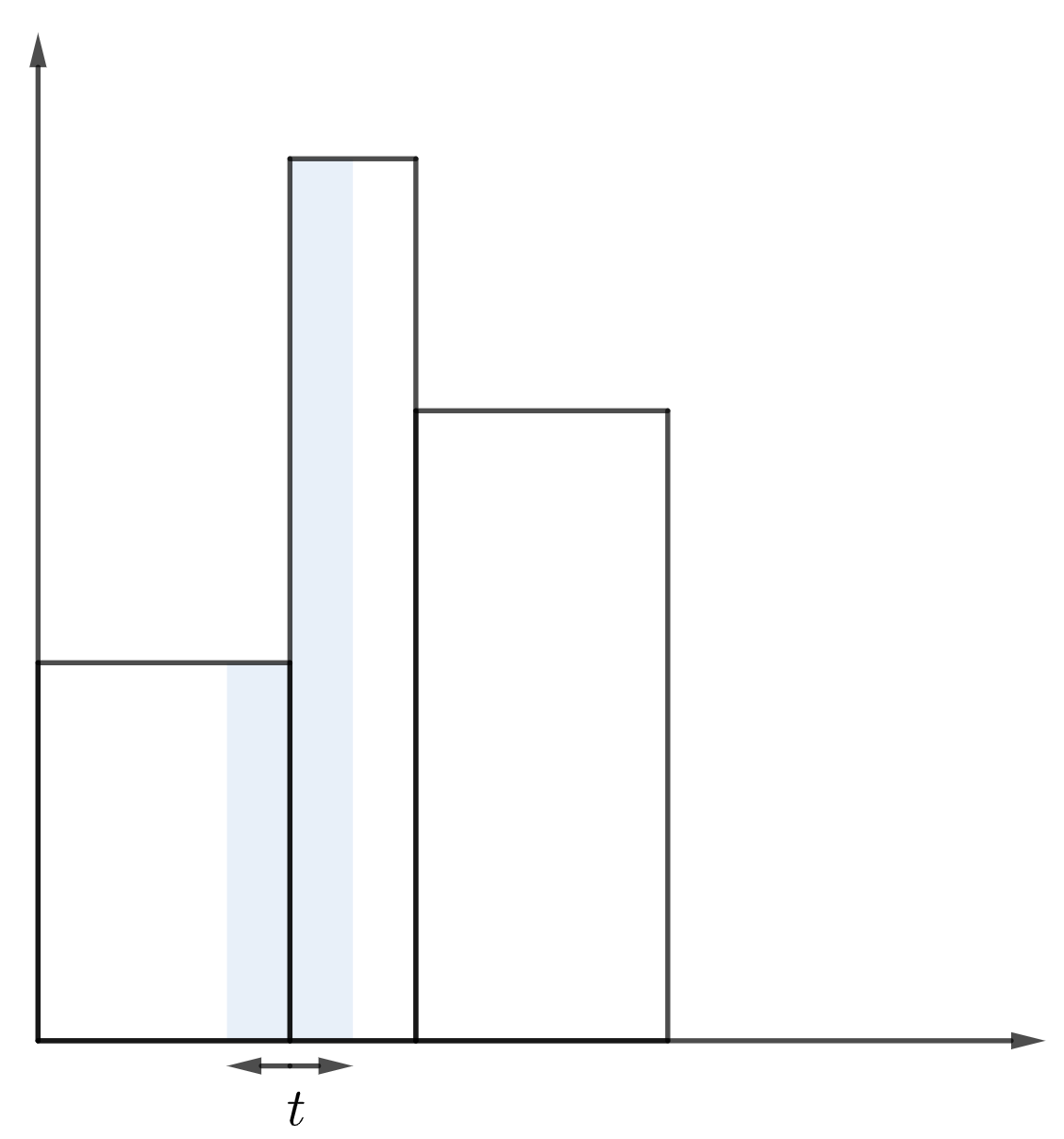}
  \end{minipage}
  \
  \begin{minipage}[b]{0.5\textwidth}
     \includegraphics[scale=0.6]{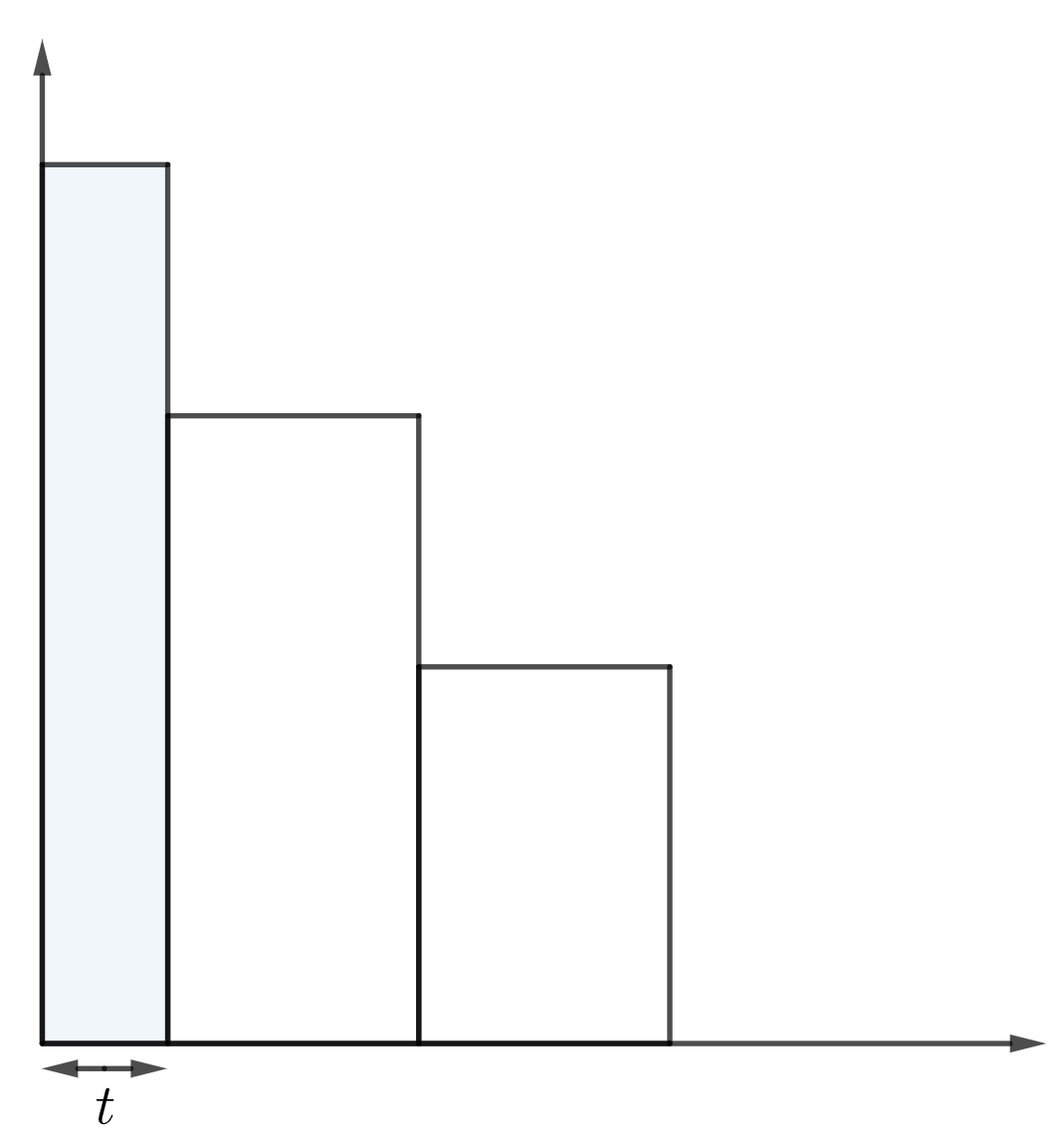}
  \end{minipage}
    \caption{For the step function $X_m$, it holds that the set $E$ with $|E|=t$ that makes $\int_E X_m^*$ biggest possible is precisely $E=[0,t]$. By construction, it happens that $\int_E X_m^* = \sup_{|E|=t}\int_E X_m$. }
    \label{fig:rearr}
\end{figure}

Now, on one hand, by definition,
\begin{align*}
X_{m,j}^*(t)&
=\inf\{s>0: |\{x\in Q_j: X_{m,j}(x)>s\}|\le t|Q_j|\}\\
&=\inf\Big\{s>0: \sum_{l>s}\frac{|Q_{j,l}|}{|Q_j|}\le t\Big\}
\end{align*}
and moreover 
$$
\sum_{l>s}\frac{|Q_{j,l}|}{|Q_j|}=\sum_{l>s} \binom{m}{l} \varepsilon_j^l (1-\varepsilon_j)^{m-l}=1-\mathbb{P}_j(X_{m,j}\leq s).
$$
On the other hand
\begin{align*}
X_{m}^*(t)&
=\inf\{s>0: |\{x\in [0,1]: {X}_{m}(x)>s\}|\le t\}\\
&=\inf\Big\{s>0: \sum_{l>s}|{X}^{-1}_m(l)|\le t\Big\}
\end{align*}
and moreover
\begin{align*}
\sum_{l>s}|{X}^{-1}_m(l)|&=\sum_{l>s} \binom{m}{l} (\min\{1,C_q/m^{1/q}\})^l (1-\min\{1,C_q/m^{1/q}\})^{m-l}\\
&= 1-\mathbb{P}({X}_m\leq s). 
\end{align*}

It is known that the cumulative distribution function of a binomial distribution $X\sim B(n,\beta)$ can be expressed as (see \cite[Chapter~VI~(10.9)]{Feller})
$$
1-\mathbb{P}(X\le k)=n\binom{n-1}{k} \int_0^{\beta}t^k(1-t)^{n-k-1}\,dt.
$$
Therefore, since $\varepsilon_j\le \min\{1,C_q/m^{1/q}\}$, we get 
\[
1-\mathbb{P}_j(X_{m,j}\leq s)\leq 1-\mathbb{P}({X}_m\leq s),
\]
and so $X_{m,j}^*(t)\leq X_m^*(t)$ for all $t\in [0,1]$. Thus,  we have
\[
 H_{m, \varepsilon_j}(t)  \leq H_{m}(t),\qquad t\in[0,1].
\]

Hence, we are left with the functions $F_m(t) := \frac{H_m(t)}{t^{1/q} m^{1-1/q}}$ and we want to estimate them uniformly in $m \in \mathbb{N}$ and $t \in (0,1]$. In view of the trivial estimate $H_m(t) \leq mt$, we may assume the case $m \geq m_0$, where $m_0$ is some large integer possibly depending on $q$. 
We first show that for large values of $q$ the claim follows from Chebyshev's inequality. More precisely, we assume that $q \geq \varphi$, where $\varphi = \frac{\sqrt{5} + 1}{2}$ is the constant representing the so-called golden ratio, i.e., $\varphi$ is the unique positive solution of the quadratic equation $x^2 - x - 1 = 0$. 

Choose $m_0$ to be the smallest $m \in \mathbb{N}$ satisfying $C_q / m^{1/q} \leq 1$ and $C_q m^{1-1/q} \geq 1$. For fixed $m \geq m_0$ the function $F_m \colon (0,1] \to (0, \infty)$ is continuous on $(0,1]$ and $\lim_{t \rightarrow 0} F_m(t) = 0$. Moreover, observe that $H_m$ is piecewise linear. In fact, since $X_m^*$ is a decreasing function, we know that the sets $\{t\in[0,1]:X_m^*(t)\geq l\}$, $l\in\{0,1,\ldots,m\}$, are of the form $[0,\alpha_{m,l}]$ for some $0<\alpha_{m,l}<1$. Then, for $\alpha_{m,l+1} < t <\alpha_{m,l}$, we have
\[
\begin{split}
H_m(t)&=\mathbb{P}(X_m^*\geq l+1)+\int_{\alpha_{m,l+1}}^t X_m^*(s)ds\\
&=\mathbb{P}(X_m^*\geq l+1)+(t-\alpha_{m,l+1})l\\
&=\alpha_{m,l+1}+(t-\alpha_{m,l+1})l, 
\end{split}
\]
and so  $H_m'(t) = l$ for these values of $t$.  This implies that   $F_m$ is a.e. differentiable and $F_m'$ is integrable. Indeed, this is because
\begin{align*}
   F_m'(t)&= \frac{H_m'(t)  - H_m(t) q^{-1} t^{-1} }{t^{1/q} m^{1-1/q}}\\
   &=\frac{1}{t^{1/q}m^{1-1/q}}\sum_{l=0}^{m-1}\Big[ l-\alpha_{m,l+1}\frac{l-1}{qt}\Big]\chi_{(\alpha_{m,l+1},\alpha_{m,l})}(t),
\end{align*}
 for every $t\in [0,1]$ which implies that we have no singularity at $0$, since $\alpha_{m,m}=0$.
Consequently,
\[
F_m(t) =\int_{0}^t F_m'(s) ds= \int_{0}^t \frac{H_m'(s) s^{1/q} - H_m(s) q^{-1} s^{-1+1/q} }{s^{2/q} m^{1-1/q}} ds \leq \int_{0}^1 \frac{H_m'(s) }{s^{1/q} m^{1-1/q}} ds.
\]
We compute the last quantity
\begin{align}\label{eq:Fm}
\begin{split}
   \int_{0}^1 \frac{H_m'(s) }{s^{1/q} m^{1-1/q}} ds 
    & = \frac{q}{q-1} \frac{1}{m^{1-1/q}}\sum_{l=0}^m l \cdot \Big (\alpha_{m,l}^{1-1/q} - \alpha_{m,l+1}^{1-1/q} \Big) \\
    & = \frac{q}{q-1} \frac{1}{m^{1-1/q}}\sum_{l=0}^m \alpha_{m,l}^{1-1/q}.
    \end{split}
\end{align}
We will finish if we estimate the sum above by a constant multiple of $m^{1-1/q}$.

Note that 
\[
\mathbb{E}(X_m^*) = C_q m^{1-1/q} \quad \text{and} \quad \text{Var}(X_m^*) = C_q m^{1-1/q} (1-C_q/m^{1/q}) \leq C_q m^{1-1/q},
\]
and recall Chebyshev's inequality 
\[
\mathbb{P} \Big( \big|X_m^* - \mathbb{E}(X_m^*) \big| \ge n \sqrt{\text{Var}(X_m^*)} \Big) \leq \frac{1}{n^2}, \qquad n \in \mathbb{N}.
\]
Observe that, in view of the above, the value $l=\big \lfloor C_qm^{1-1/q} + \sqrt{C_qm^{1-1/q}} \big \rfloor$ is the highest value of $X_m^*$ for which an application of Chebyshev's inequality gives at best the trivial estimate:
\[
\alpha_{m,l}=\mathbb{P}(X_m^*\geq l) \leq 1.
\]
For the remaining values of $X_m^*$ Chebyshev's inequality gives better estimates for the corresponding quantities $\alpha_{m,l}$, and then we have
\begin{align} \label{eq:secondmoment}
\begin{split}
\sum_{l=0}^m \alpha_{m,l}^{1-1/q}
= & \ \sum_{l = 0}^{\big \lfloor C_qm^{1-1/q} + \sqrt{C_qm^{1-1/q}} \big \rfloor } \mathbb{P}(X_m^*\geq l)^{1-1/q} \\
& + \sum_{n \in \mathbb{N}} \sum_{l = \big \lfloor C_qm^{1-1/q} + n \sqrt{C_qm^{1-1/q}} \big \rfloor + 1}^{\big \lfloor C_qm^{1-1/q} + (n+1) \sqrt{C_qm^{1-1/q}} \big \rfloor} \mathbb{P}(X_m^*\geq l)^{1-1/q}.
\end{split} 
\end{align}
For the first sum, we bound each term by $1$. For the second sum, Chebyshev's estimate is used to bound each term in the inner sum by $n^{-2+2/q}$. Furthermore, we take into account the fact that $ \mathbb{P}(X_m^*\geq l)=0$ for each $l > m$, which implies that the outer sum is not actually an infinite sum, but a finite sum from $n=1$ to a value of $n$ such that $ n \sqrt{C_qm^{1-1/q}}  \geq m$ (and therefore $\big \lfloor C_qm^{1-1/q} + n \sqrt{C_qm^{1-1/q}} \big \rfloor \geq m$), that is, from $n=1$ to $n = \big\lfloor \sqrt{C_q^{-1}m^{1+1/q}} \big\rfloor$. Thus, we get the estimate
\begin{align} \label{eq:secondmomentb}
\begin{split}
\sum_{l=0}^m \alpha_{m,l}^{1-1/q}
 & \leq  \ 3C_q m^{1-1/q} +  2 \sqrt{C_qm^{1-1/q}} \sum_{n = 1}^{\big \lfloor \sqrt{C_q^{-1} m^{1+1/q}} \big\rfloor}n^{-2+2/q}.
\end{split} 
\end{align}
Note that the second term in the above sum can be estimated by 
\begin{multline*}
 2 \sqrt{C_qm^{1-1/q}}\Big(1+\int_1^{\lfloor \sqrt{C_q^{-1} m^{1+1/q}} \rfloor}x^{-2+2/q}\,dx\Big)\\
 \leq C(q) \sqrt{C_qm^{1-1/q}}+ m^{\frac{1}{2} - \frac{1}{2q} } m^{(-1+\frac{2}{q})(\frac{1}{2} + \frac{1}{2q})}
\end{multline*}
for some suitable constant $C(q) < \infty$.
For $q \geq \varphi$ we have $\frac{1}{2} - \frac{1}{2q} + (-1+\frac{2}{q})(\frac{1}{2} + \frac{1}{2q}) \leq 1-\frac{1}{q}$, which implies that the last sum can be estimated by a constant multiple of $m^{1-1/q}$. Then we got the desired estimate for the sum in \eqref{eq:Fm}, and we conclude that, in this case, $F_m(t)$ is bounded uniformly in $m$ and $t$. This gives that $M^{\mathcal{R}'}$ is indeed of restricted weak type $(q,q)$. 

To obtain the claim in the remaining range $q \in (1, \varphi)$, 
we have to recall the estimates for higher order moments of binomial random variables. Given $q \in (1, \varphi)$ we choose $R \in \mathbb{N}$ such that $R-R/q > 1$. Moreover, we can take $m_0$ so large that for $m \geq m_0$ the inequalities $m\ge R$ and $R / (C_q m^{1-1/q}) < e$ are fulfilled. Thus, using the bounds by Lata\l{}a~\cite{La1997} (see also \cite[Proposition 11]{Ja2019}) one can show that
\[
\big(\mathbb{E}((X_m^*)^{R}) \big)^{1/R} \leq C_{q,R} m^{1 - 1/q}, \qquad m \in \mathbb{N},
\]  
holds for some suitable constant $C_{q,R} < \infty$ (for the proof see \cite[Proposition 11]{Ja2019}). 
Thus, by Markov's inequality stated for $(X_m^{*})^R$,
\[
\mathbb{P} \Big( (X_m^{*})^R \ge (n C_{q,R} m^{1 - 1/q})^{R} \Big) \leq \frac{1}{n^{R}}, \qquad n \in \mathbb{N}.
\]
We can also assume that $m$ is so large that $C_{q,R} m^{1 - 1/q} \geq 1$, so we can argue as in the previous case to get
\begin{align*} 
\sum_{l=0}^m \alpha_{m,l}^{1-1/q} & =  \sum_{l = 0}^{\lfloor C_{q,R} m^{1 - 1/q} \rfloor} \mathbb{P}((X_m^{*})^R\geq l^R)^{1-1/q} \\
&\quad + 
\sum_{n \in \mathbb{N}} \sum_{l = \lfloor n C_{q,R} m^{1 - 1/q} \rfloor + 1}^{\lfloor (n+1) C_{q,R} m^{1 - 1/q} \rfloor} \mathbb{P}((X_m^{*})^R\geq l^R)^{1-1/q} \\
& \leq 2C_{q,R} m^{1-1/q} + \sum_{n = 1}^{\infty} 2 C_{q,R} m^{1-1/q} n^{-R+R/q}.
\end{align*}
The last quantity is bounded by a constant multiple of $m^{1-1/q}$, so we again got the desired estimate for the sum in \eqref{eq:Fm}, and the proof is completed.
\end{proof}

\begin{proof}[Proof of Corollary~\ref{corollary2}] 
 
{
We prove first item (1), namely the unboundedness in the case $q \in [1, q_0)$. To that end, we take a family of cubes $\{Q_j\}_{j\in\mathbb{N}}\subset \mathcal{R}$ and   $(\varepsilon_j, l_j)$-configurations around the cubes $Q_j$ as in Example~\ref{ex:epsn}, with elements contained in disjoint sets $E_j$ and satisfying $|Q_j \cap Q_j^{(k)}| = (1 - \varepsilon_j) |Q_j|$, where $\varepsilon_j = \frac{1}{j+1}$ and $l_j = \lfloor j^{q_0} \rfloor$ if $q_0 \in (1, \infty)$. 

For item (2) we will prove the unboundedness in the case $q \in [1, q_0]$. We use $\varepsilon_j$ as before and $l_j = \lfloor \log(j+2) j^{q_0} \rfloor$ (the $\log$ term ensures unboundedness when $q=q_0$).

In any of the above cases, an application of Theorem~\ref{thm:r_0} gives the result.
}
\end{proof}
%%%%%%%%%%%%%%%%%%%%%%%%%%%%
\subsection{Differentiation properties of bases}
\label{sub:differentiation}
%%%%%%%%%%%%%%%%%%%%%%%%%%%%

A natural parallel problem is to study differentiability properties of bases $\mathcal{R}_0\subset \mathcal{R}'\subset \mathcal{R}$. 

\begin{definition}(\cite[Section~6.1]{bruckner}, \cite[Ch.~2]{guzmanyellow}.)
For every $y\in \tom$ let $\mathcal{B}(y)$ be a collection of measurable sets of positive measure that contain (or whose topological closures contain) the point $y$. If $\{S_n\}_n\subset \mathcal{B}(y)$ and $\delta(S_n)\to 0$, then we say that a sequence $S_n$ \textit{contracts to} $y$, and write $S_n\Rightarrow y$. 
Suppose that for each $y \in \tom$ there exists a sequence $\{S_n\}_n\subset \mathcal{B}(y)$ such that $S_n\Rightarrow y$. Let $\mathcal{B}:=\cup_{y\in\tom}\mathcal{B}(y)$. We call $(\mathcal{B},\Rightarrow)$ a \textit{differentiation basis}.
\end{definition}

Let $(\mathcal{B},\Rightarrow)$ be a differentiation basis in $\tom$. Given $f\in L^1(\tom)$, we define the \emph{upper and lower derivative} of $\int f$ with respect to $\mathcal{B}$ (and the Haar measure $dx$) in the point $x\in\tom$ by (without loss of generality we assume here that $f$ is a real function)
$$
\overline{D}\Big( \int f,x \Big)=\sup_{\substack{\{B_n\}\subset\mathcal{B}\\B_n\Rightarrow x }}\Big\{ \limsup_{n} 
f_{B_n}\Big\}
\quad\text{ and }\quad
\underline{D}\Big( \int f,x \Big)=\inf_{\substack{\{B_n\}\subset\mathcal{B}\\B_n\Rightarrow x }}\Big\{ \liminf_{n} f_{B_n}\Big\},
$$
respectively. When,  {for a set $A$},
\begin{equation}\label{Bdiferintf}
\overline{D}\Big(\int f,x \Big)=\underline{D}\Big( \int f,x \Big)
=f(x)\quad \text{ a.e. } {x\in A}
\end{equation}
holds, we write $D\big(\int f,x \big)=f(x)$ and say that the basis $\mathcal{B}$ \emph{differentiates} $\int f$  {in $A$} and that the \emph{derivative} of $\int f$ is $f$. A necessary condition for \eqref{Bdiferintf} is that 
\begin{equation*}
\lim_{n\in\mathbb{N}} f_{B_n} =f(x)\quad \text{ a.e. }
\end{equation*}
holds, for every sequence $\{B_n\}_{n\in\mathbb{N}}\subset  \mathcal{B}$ such that $B_n\Rightarrow x$. When \eqref{Bdiferintf} is satisfied for all $f\in L^\infty(\tom)$ (resp. $f\in L^1(\tom)$), we say that $\mathcal{B}$ \emph{differentiates} $L^\infty(\tom)$ (resp. $L^1(\tom)$). 
Note that $L^\infty(\tom)\subset L^1(\tom)$ and thus, if the basis $\mathcal{B}$ \emph{does not} differentiate $L^\infty(\tom)$, then it also does not differentiate  $L^1(\tom)$.

 In the classical context of $\mathbb{R}^n$, for a homothecy invariant basis $\mathcal{B}$, differentiation of $L^1$ is equivalent to weak $(1,1)$ boundedness of $M^\mathcal{B}$. See \cite[Thm.~1.1]{guzmanywelland}.

In our context, the following result holds. Its proof is standard, so we omit it.
\begin{theorem}\label{thm:weak->diff}
Let $\mathcal{B}$ be a differentiation basis in $\mathbb{T}^\omega$. If the operator $M^{\mathcal{B}}$ is of weak type $(1,1)$, then the basis  $\mathcal{B}$ does differentiate $L^1(\mathbb{T}^\omega)$.
\end{theorem}

Note that the previous theorem is only one half of the version in $\mathbb{R}^n$. This is in general the best possible result, since
the first author \cite[Proposition~4.1]{Kosz2020} found a basis $\mathcal{D}$ for which the associated maximal operator $M^\mathcal{D}$ is not of weak type $(1,1)$ but it still differentiates $L^1(\tom)$. Nevertheless, this basis $\mathcal{D}$ is not contained in $\mathcal{R}$. Our results in Theorems~\ref{unweighted_unboundedness}~and~\ref{thm:r_0} provide an example of a basis $\mathcal{R}'$ satisfying $\mathcal{R}_0\subset\mathcal{R}'\subset \mathcal{R}$ and such that $M^{\mathcal{R}'}$ is not of weak type $(1,1)$. An immediate question is in order: does $\mathcal{R}'$ differentiate $L^1(\tom)$? In this regard, we can parallelly ask about the threshold differentiation basis.
 
In the following discussion, we will check that, within the bases $\mathcal{R}'$, we may find examples of differentiation bases and of bases which do not differentiate $L^1(\tom)$. We restrict our attention to the case of configurations described in Example~\ref{ex:epsn}, i.e., configurations  {$S_j$} around some cubes $Q_j$ which we specify below. 
 
On the one hand, observe that for each $j$, all the elements of a configuration $S_j$ around a cube $Q_j$, together with $Q_j$ itself, are contained in a set $E_j \subset \tom$ of size $(l_j+1) |Q_j|$  {(see Figure \ref{fig:el-conf})}. If the sequence $\{Q_j\}_j$  satisfies  $\sum_{j} (l_j+1) |Q_j| < \infty$ (in particular, this happens if the sequence of sizelevels $\{m_{Q_j}\}_{j \in \mathbb{N}}$ is strictly increasing, since $l_j \leq \ell(m_{Q_j}) \leq m_{Q_j}$), then a similar strategy as in \cite[Proposition 4.1]{Kosz2020} can be used to show that  {the basis} $\mathcal{R}'$  {defined as $\mathcal{R}':=\mathcal{R}_0\cup S$, where $S=\bigcup_{j\ge1}S_j$,}
differentiates $L^1(\tom)$. Indeed, for each $\delta>0$ find $J \in \mathbb{N}$ such that $|\bigcup_{j > J} E_j| < \delta$.

If $f \in L^1(\tom)$, then the basis $\mathcal{R}'$ differentiates $f$ in $( \bigcup_{j > J} E_j )^c$. This can be seen as follows: for a given $x\in ( \bigcup_{j > J} E_j )^c$, the sets from $S$ with $j > J$ are not taken into account in checking the differentiation condition. On the other hand, the remaining elements of $\mathcal{R}'$ belong to a family of the form $\mathcal{R}_0\cup F$ with $F$ being a finite collection of cubes. Hence, an application of Lemma \ref{lematamanos} gives us the weak type $(1,1)$ bound for the associated maximal operator. Finally, Theorem \ref{thm:weak->diff} implies good differentiation properties for $\mathcal{R}'$. Since $\delta > 0$ was arbitrary, we are done.

 On the other hand, nothing can be said in general if the condition $m_{Q_{j+1}} > m_{Q_{j}}$ is not satisfied. If $\sharp(\{ j \in \mathbb{N} : m_{Q_{j}}=n \})$ grows like $n^{\alpha}$ for some $\alpha >0$, then we have $\sum_{j} (l_j+1) |Q_j| < \infty$, as previously. However, for $m_{Q_{j}}$ growing to infinity very slowly it may happen that $\mathcal{R}'$ does not differentiate $L^1(\tom)$.
 
 We have an example from \cite[Proposition 3.2]{Kosz2020}, where it is shown that $\mathcal{R}$ does not differentiate $L^1(\tom)$ by finding $\widetilde{f} \in L^1(\tom)$, $\widetilde{f} \geq 0$, and $E \subset \tom$, $|E| > 0$, such that for each $x \in E$ we have $\limsup_{n \rightarrow \infty} \widetilde{f}_{P_n} > \widetilde{f}(x)$ for some $x \in P_n \in \mathcal{R}$, $|P_n| \rightarrow 0$. Moreover, careful reading of the proof of \cite[Proposition 3.2]{Kosz2020} reveals that all the sets $P_n$ can be selected from a countable subset of $\mathcal{R}_0$, say $\mathcal{P}$. Thus, one can  arrange the cubes added to $\mathcal{R}_0$ in such a way that $\mathcal{R}' \supset \mathcal{P}$, making the differentiation condition impossible to hold. Indeed, assuming $\mathcal{P}=\{P_1, P_2, \dots\}$ one may choose arbitrary configurations $S_j$, $j \in \mathbb{N}$, such that $P_j \in S_j$, and the postulated inclusion is then satisfied.

\section{Weighted setting} \label{sec:pesos}

As explained in Section~\ref{sec:Intro}, in view of the negative results concerning the $L^q$ boundedness of the maximal function associated with $\mathcal{R}$, it seems natural to  introduce weights and study weighted boundedness properties of the corresponding maximal function. The previous results motivate to study maximal functions related to bases $\mathcal{R'}$ described in Theorem~\ref{thm:r_0}, for which the unweighted unboundedness is characterized. However, we opted to focus on the more tractable basis $\mathcal{R}$, as a first step to understand the situation.

In the Euclidean setting, the maximal operator defined with respect to the basis consisting of all cubes in $\mathbb{R}^n$ is well understood in the light of the theory of weights. Beyond the Euclidean setting,  Duoandikoetxea, Mart\'in-Reyes and Ombrosi \cite{Duoandikoetxea2001} showed that, in general, several {\it a priori} different classes of weights can be defined in relation with a given basis $\mathcal{B}$ and that there are alternative characterizations of some of such classes of weights. In this regard, recall the classes of weights introduced in Definition~\ref{def:weights1}. We have the following.

\begin{theorem}[{cf.\ \cite[Theorem~3.1]{Duoandikoetxea2016}}]\label{teo:weights2}
Let $w\in L^1(\tom)$ be a weight and let $\mathcal{B}$ be a basis of sets in $\tom$. 
\begin{enumerate}
\item The weight $w$ is a Muckenhoupt weight with respect to $\mathcal{B}$ if and only if there are constants $\rho,C>0$ such that for every $Q\in\mathcal{B}$ and every measurable subset $E\subset Q$ it holds that
\[
\frac{|E|}{|Q|}\leq C\left(\frac{w(E)}{w(Q)}\right)^\rho.
\]
Moreover, if $w\in A_p^\mathcal{B}$ for some $p \in (1, \infty)$, then $\rho$ can be taken to be $1/p$.
\item The weight $w$ satisfies a reverse H\"older inequality with respect to $\mathcal{B}$ if and only if there are constants $\delta,C>0$ such that for every $Q\in\mathcal{B}$ and every measurable subset $E\subset Q$ it holds that
\[
\frac{w(E)}{w(Q)}\leq C\left(\frac{|E|}{|Q|}\right)^\delta.
\]
Moreover, if $w\in \mathrm{RH}_r^\mathcal{B}$ for some $r \in (1, \infty)$, then $\delta$ can be taken to be $1/r'$.
 
\end{enumerate}
\end{theorem}

This result is a particular case of \cite[Theorem~3.1]{Duoandikoetxea2016}. Also, when having at hand a local Calder\'on--Zygmund decomposition consisting of base sets of $\mathcal{B}$, it is a standard fact that the classes of Muckenhoupt weights and reverse H\"older weights are the same, see \cite[Chapter~7]{Duoandikoetxea2001}.  This is the case for the Rubio de Francia basis $\mathcal{R}$. Indeed, due to the good structure of the restricted Rubio de Francia basis  {$\mathcal{R}_0$} it is possible to prove, by using standard arguments, a decomposition of Calder\'on--Zygmund type localized to a given $Q\in \mathcal{R}$, by using a dyadic decomposition of $Q$ which we can refer to as $\mathcal{R}_0(Q)$ (for the details see \cite[Theorem~9]{Fernandez-Roncal2020}). Thus, again with standard arguments, we can conclude the following.

\begin{corollary}\label{cor:RH=Ap}
The class of Muckenhoupt weights with respect to $\mathcal{R}$ and the class of reverse H\"older weights with respect to $\mathcal{R}$ coincide.
\end{corollary} 

\subsection{Examples of Muckenhoupt weights with respect to \texorpdfstring{$\mathcal{R}$}{R}}\label{examplesofweights}

In this subsection we provide examples of Muckenhoupt weights with respect to the Rubio de Francia basis.  Precisely, we will show two easy examples, regarding cylindrical weights (i.e., weights depending only on a finite number of variables) and tensor products. Actually, the former can be regarded as a particular case of the latter, but we show both cases to make the presentation more instructive. 

Consider a weight  $w\in A_1(\mathbb{T})$, i.e., a weight satisfying  
\[
[w]_{A_1(\mathbb{T})}:=\esssup_{t\in\mathbb{T}}\frac{Mw(t)}{w(t)}<\infty.
\] Here $Mw$ denotes the Hardy--Littlewood maximal function in the torus $\mathbb{T}$, that is,
\[
Mw(t) := \sup_{I\ni t} w_I,
\]
where the supremum is taken among all intervals  $I\subset\mathbb{T}$ containing $t$.
 There exists then a zero measure set $E\subset\mathbb{T}$ such that, for any $t\in \mathbb{T}\backslash E$,
\[
\sup_{I\ni t} w_I \leq [w]_{A_1(\mathbb{T})} w(t).
\]
Define the function $v \colon \tom\to \mathbb{R}$ given by $v(x)=w(x_1)$. Fubini's theorem  ensures that if $f\in L^1(\tom)$, then  {for any $n\in\mathbb N$}, 
$$
\int_{\tom}f(x)\,dx=\int_{\mathbb{T}^n}\int_{\mathbb{T}^{n,\omega}}f(x_{(n},x^{(n})\,dx^{(n}\,dx_{(n}=\int_{\mathbb{T}^{n,\omega}}\int_{\mathbb{T}^n}f(x_{(n},x^{(n})\,dx_{(n}\, dx^{(n},
$$
where the inside integrals exist a.e. and they lead to integrable functions in the rest of the variables.
Thus $v(x)$ is a weight function which satisfies  
\[
\begin{split}
M^\mathcal{R}v(x)=\sup_{  \mathcal{R}\ni Q\ni x} v_Q &=\sup_{ \mathcal{R}\ni Q\ni x}\frac{1}{|Q|}\int_{\pi_{\mathbb{T}^{1,\omega}}
	(Q)}\int_{\pi_{\mathbb{T}}(Q)} w(y_1) dy_{(1} dy^{(1}\\
&= \sup_{ \mathcal{R}\ni Q\ni x}\frac{1}{|\pi_{\mathbb{T}^{1,\omega}}(Q)|}\int_{\pi_{\mathbb{T}^{1,\omega}}(Q)}\frac{1}{|\pi_{\mathbb{T}}(Q)|}\int_{\pi_{\mathbb{T}}(Q)} w(y_1) dy_{1} dy^{(1}\\
&\leq  [w]_{A_1(\mathbb{T})}w(x_1) \sup_{ \mathcal{R}\ni Q\ni x}\frac{1}{|\pi_{\mathbb{T}^{1,\omega}}(Q)|}\int_{\pi_{\mathbb{T}^{1,\omega}(Q)}}  dy^{(1} \\
&=[w]_{A_1(\mathbb{T})}v(x)
\end{split}
\]
for every $x\in \tom$ with $x_1\in\mathbb{T}\backslash E$, where $\pi_S(Q)$ denotes the projection of $Q$ on $S$ (this should not be confused with the measure of $Q$ induced by the function $\pi_S$). Since $ E\times \mathbb{T}^{1,\omega}$ is a zero measure subset of $\tom$, we have that $v$ is an $A_1^\mathcal{R}(\tom)$ weight. 
As $A_1^\mathcal{R}(\tom)\subset A_p^\mathcal{R}(\tom)$ for each $p \in (1, \infty)$,  the weight $v$ is also in $A_p^\mathcal{R}(\tom)$.  
 
More generally, the preceding construction can be generalized to include those Muckenhoupt weights of the $n$-dimensional torus which are associated with the strong maximal function, that is, the class of $A_p^\mathcal{B}(\mathbb{T}^n)$ weights where $\mathcal{B}$ consists of all rectangles in $\mathbb{T}^n$. Thus the  classes of weights we are working with are not empty and moreover they share some elements.

 A similar argument can be performed for weights of the form $v(x):=\prod_{j=1}^\infty w_j(x_j)$ with $\{w_j\}_{j\in\mathbb N}$ being a sequence of weights defined on $\mathbb T$. To ensure that the above product converges a.e. to a reasonable weight in $\tom$ let us assume that $w_j(x_j) = 1$ for $x_j \in A_j \subset \mathbb{T}$ with $\sum_{j \in \mathbb{N}} |\mathbb{T} \setminus A_j| < \infty$, and that $\Pi_{j \in \mathbb{N}} w_j(\mathbb{T}) \in (0,\infty)$. Suppose that $\prod_{j=1}^\infty[w_j]_{A_p(\mathbb{T})}<\infty$, for some $p \in (1, \infty)$, where
 \[
 [w]_{A_p(\mathbb{T})}=  \sup_{I}\left(\frac{1}{|I|}\int_I w(x)dx\right)\left(\frac{1}{|I|}\int_I w(x)^{1-p'}dx\right)^{p-1}<\infty
 \]
 (here the supremum is taken among all intervals $I\subset\mathbb{T}$). Then $v$ turns out to be a tensor weight which belongs to the Muckenhoupt class introduced in Definition~\ref{def:weights1}. Similarly, if $w_j\in \mathrm{RH}_r(\mathbb{T})$ for all $j\in\mathbb{N}$ and $\prod_{j=1}^\infty[w_j]_{\mathrm{RH}_r(\mathbb{T})}<\infty$ for some $r>1$, then $v$ is a tensor weight belonging to the $\mathrm{RH}_r^\mathcal{R}(\tom)$ class.

 %%%%%%%%%%%
\subsection{Construction of weights in \texorpdfstring{$\tom$}{Tom} via periodization procedure}
\label{subsec:const}
%%%%%%%%%%%%%%

Many examples are at hand by periodizing weights in the real line. This can be done by taking into account that any weight $w\in L^1_{\mathrm{loc}}(\mathbb{R})$ induces a weight in $\ell^1$ (recall that a weight in a measure space $X$ is a nonnegative locally integrable function, i.e., integrable on compact sets), which we will denote by $w$ as well and which is defined by $w(k)=w([0,1]+k)$ for all $k\in\mathbb{Z}$.
\begin{definition}
Let $w\in L^1_{\mathrm{loc}}(\mathbb{R})$ be a weight.  For every nonnegative sequence $\{a_k\}_{k\in\mathbb{Z}}\in \ell^1(w)$ we define the periodization ${\bf p}[w,\{a_k\}_{k\in\mathbb{Z}}]$ of $w$ by 
\[
{\bf p}[w,\{a_k\}_{k\in\mathbb{Z}}](x)=\sum_{k\in\mathbb{Z}} a_kw(x+k),\qquad x\in \mathbb{T}.
\]
\end{definition}
This is a way to compress the information about the weight $w$ in the interval $[0,1]$ by downplaying values corresponding to points which are far away from the origin. The periodization ${\bf p} [w,\{a_k\}_{k\in\mathbb{Z}}]$ of a weight $w\in L^1_{\mathrm{loc}}(\mathbb{R})$ is a weight in $\mathbb{T}$. This is immediate from the summability of $\{a_k\}_{k\in\mathbb{Z}}$ in the $\ell^1(w)$ sense. Indeed, the nonnegativity is clear and for the integrability we have that
\begin{align*}
\int_{\mathbb{T}} {\bf p}[w,\{a_k\}_{k\in\mathbb{Z}}](x)dx = \sum_{k\in\mathbb{Z}}a_k\int_{[0,1]} w(x+k) \, dx
=\sum_{k\in\mathbb{Z}}a_kw([0,1]+k)<\infty.
\end{align*}
We point out that other periodization procedures have been considered in the literature, see for instance the recent article \cite{Berkson2020}.

\begin{example}\label{log}
We present two simple examples of weights in $\mathbb{T}$, which arise as a result of the periodization procedure described above.

\begin{enumerate}
\item Firstly, we have that ${\bf p}[w,\{2^{-|k|}\}_{k\in\mathbb{Z}}]$ is a weight in $\mathbb{T}$ for any power weight $w$ in $\mathbb{R}$. Indeed, if $w(x)=|x|^\alpha$, $\alpha>-1$, then $w([0,1]+k)\leq w([-|k|-1,|k|+1]) = \frac{2(|k|+1)^{\alpha+1}}{\alpha+1}$, and the latter quantities form a sequence which is summable in $k \in \mathbb{Z}$ after multiplication term by term by $\{2^{-|k|}\}_{k\in\mathbb{Z}}$.

\item Secondly, let us consider the weight $w(x)=\max \{\log\frac{1}{|x|},1 \}\in L^1_{\mathrm{loc}}(\mathbb{R})$. The periodization ${\bf p}[w,\{2^{-|k|}\}_{k\in\mathbb{Z}}]$ is a weight in $\mathbb{T}$. Indeed, as for power weights, we can perform the bound $w([0,1]+k)\leq w([-|k|-1,|k|+1])$ and the latter quantity is easily computed, thus getting $2|k|+2+2/e$,
which is again summable in $k \in \mathbb{Z}$ after multiplication term by term by $\{2^{-|k|}\}_{k\in\mathbb{Z}}$.
\end{enumerate}
\end{example}

Some properties of weights in the real line are conserved via periodization.
\begin{proposition}
\label{prop:perioA1}
Let $w$ be an $A_1(\mathbb{R})$ weight. If $\{\lambda^{-|k|}\}_{k\in\mathbb{Z}}\in\ell^1(w)$, $\lambda>1$, then the periodization ${\bf p}[w,\{\lambda^{-|k|}\}_{k\in\mathbb{Z}}]$ is an $A_1(\mathbb{T})$ weight and $[w]_{A_1(\mathbb{T})}\leq \lambda^2[w]_{A_1(\mathbb{R})}$.
\end{proposition}

\begin{proof}
Let $I\subset \mathbb{T}$. There are two possibilities for $I$. One of them is that its representative in the interval $[0,1]$ has just one connected component. In this case,
\begin{align*}
    \frac{1}{|I|}\int_I {\bf p}\big[w,\{\lambda^{-|k|}\}_{k\in\mathbb{Z}}\big](z)dz
    &= \frac{1}{|I|}\int_I\sum_{k\in\mathbb{Z}}\lambda^{-|k|}w(z+k)dz\\
    &= \sum_{k\in\mathbb{Z}} \lambda^{-|k|}\frac{1}{|I+k|}\int_{I+k} w(z)dz\\
    &\leq [w]_{A_1(\mathbb{R})}\sum_{k\in\mathbb{Z}} \lambda^{-|k|}\essinf_{z\in I }w(z+k)\\
    &\leq [w]_{A_1(\mathbb{R})}\sum_{k\in\mathbb{Z}}\lambda^{-|k|} w(x+k) \\
    & =  {[w]_{A_1(\mathbb{R})}} \, {\bf p}\big[w,\{\lambda^{-|k|}\}_{k\in\mathbb{Z}}\big](x)
\end{align*}
for a.e. $x\in I$.

In case the representative of $I$ in the interval $[0,1]$ has two connected components, we write $I^-$ and $I^+$ respectively for the left hand side and the right hand side of the representative of $I$ inside $[0,1]$. In this situation, we have that $\widetilde{I} = (I-1)^+\cup I^-$ is a representative of $I$ in $\mathbb{R}$ with just one connected component. We can then estimate
\begin{align*}
\frac{1}{|I|}\int_I  {\bf p}\big[w,\{\lambda^{-|k|}\}_{k\in\mathbb{Z}}\big](z)dz 
& = \frac{1}{|I|}\int_{I} \sum_{k\in\mathbb{Z}} \lambda^{-|k|}w(z+k)dz\\
&\leq \sum_{k\in\mathbb{Z}} \lambda^{-|k|+1}\frac{1}{|I|}\int_{\widetilde{I}} w(z+k)dz\\
&\leq [w]_{A_1(\mathbb{R})} \sum_{k\in\mathbb{Z}} \lambda^{-|k|+1} \essinf_{z\in \widetilde{I}}w(z+k)\\
&\leq   [w]_{A_1(\mathbb{R})} \sum_{k\in\mathbb{Z}} \lambda^{-|k|+2} w(x+k) \\
& \leq \lambda^2 [w]_{A_1(\mathbb{R})} {\bf p}\big[w,\{\lambda^{-|k|}\}_{k\in\mathbb{Z}}\big](x),
\end{align*}
which again holds for a.e. $x \in I$.
\end{proof}

\noindent The weight ${\bf p} [w,\{2^{-|k|}\}_{k\in\mathbb{Z}} ]$ in Example~\ref{log} is then an $A_1(\mathbb{T})$ weight, in view of the above proposition and the fact that $w(x)=\max \big\{\log\frac{1}{|x|},1 \big\}\in A_1(\mathbb{R})$.

\begin{proposition}
\label{prop:perioRH}
Let $w$ be a $\mathrm{RH}_r(\mathbb{R})$ weight for some $r \in (1, \infty)$.   If $\{\lambda^{-|k|}\}_{k\in\mathbb{Z}}\in\ell^1(w)$, $\lambda>1$, then the periodization ${\bf p}[w,\{\lambda^{-|k|}\}_{k\in\mathbb{Z}}]$ is a $\mathrm{RH}_r(\mathbb{T})$ weight and $[w]_{\mathrm{RH}_r(\mathbb{T})}\leq  \lambda^2 [w]_{\mathrm{RH}_r(\mathbb{R})}$.
\end{proposition}

\begin{proof}
As above, for an interval $I\subset \mathbb{T}$, there are two possibilities for $I$. In case its representative in the interval $[0,1]$ has just one connected component, we can apply Minkowski's inequality and the reverse H\"older property of $w$ to get
\begin{align*}
\Big(\frac{1}{|I|}\int_I {\bf p}\big[w,\{\lambda^{-|k|}\}_{k\in\mathbb{Z}}\big](x)^rdx\Big)^{1/r}& =\Big[\frac{1}{|I|}\int_I    \Big(\sum_{k\in\mathbb{Z}}\lambda^{-|k|} w(x+k)\Big)^r dx\Big]^{1/r}\\
&\leq  [w]_{\mathrm{RH}_r(\mathbb{R})} \frac{1}{|I|}\int_{I}\sum_{k\in\mathbb{Z}}\lambda^{-|k|}      w(x+k) dx\\
&= [w]_{\mathrm{RH}_r(\mathbb{R})} \frac{1}{|I|}\int_{I}{\bf p}\big[w,\{\lambda^{-|k|}\}_{k\in\mathbb{Z}}\big](x) dx.
\end{align*}

In case the representative of $I$ in $[0,1]$ has two connected components, $I^-$ and $I^+$, we denote $\widetilde{I} = (I^+ - 1) \cup I^- $ and proceed as before to get

\begin{align*}
\Big(\frac{1}{|I|}\int_I {\bf p}\big[w,\{\lambda^{-|k|}\}_{k\in\mathbb{Z}}\big](x)^r\,dx\Big)^{1/r} 
& = \Big[\frac{1}{|I|}\int_I \Big(\sum_{k\in\mathbb{Z}} \lambda^{-|k|}w(x+k)\Big)^rdx\Big]^{1/r}\\
& \leq \sum_{k\in\mathbb{Z}} \lambda^{-|k|+1}\Big(\frac{1}{|I|}\int_{\widetilde{I}} w(x+k)^rdx\Big)^{1/r}\\
& \leq [w]_{\mathrm{RH}_r(\mathbb{R})} \sum_{k\in\mathbb{Z}} \lambda^{-|k|+1} \frac{1}{|I|}\int_{\widetilde{I}} w(x+k) dx\\
& \leq [w]_{\mathrm{RH}_r(\mathbb{R})} \sum_{k\in\mathbb{Z}} \lambda^{-|k|+2} \frac{1}{|I|}\int_{I} w(x+k) dx\\
& \leq \lambda^2 [w]_{\mathrm{RH}_r(\mathbb{R})} \frac{1}{|I|}\int_{I} {\bf p}\big[w,\{\lambda^{-|k|}\}_{k\in\mathbb{Z}}\big](x) dx.
\end{align*}
This proves the desired result. 
\end{proof}

\noindent The weight $w(x)=\max \{\log\frac{1}{|x|},1 \}$ is known to satisfy $w\in A_1(\mathbb{R})\cap (\bigcap_{r>1}\mathrm{RH}_r(\mathbb{R}))$ (this is because $w^s\in A_1(\mathbb{R})$ for all $s\geq 1$, see  \cite[p. 2948]{cruzuribe}). Then, in view of the above result, the periodization ${\bf p}[w,\{2^{-|k|}\}_{k\in\mathbb{Z}}]$ introduced in Example~\ref{log} is in $\mathrm{RH}_r(\mathbb{T})$ for every $r \in (1, \infty)$. We then have an example of a weight belonging to $A_1^\mathcal{R}(\tom)\cap (\bigcap_{r>1}\mathrm{RH}_r^\mathcal{R}(\tom))$, just by applying the procedure depicted at the beginning of this subsection. 
 
\subsection{Weighted unboundedness of \texorpdfstring{$M^\mathcal{R}$}{MR}. Proof of Theorem~\ref{thm:unboundedweights_new}} \label{weighted_unb}
  
We now prove Theorem~\ref{thm:unboundedweights_new} on weighted unboundedness for the maximal operator $M^\mathcal{R}$. The proof relies on a slight modification of the argument used in the proof of Theorem~\ref{unweighted_unboundedness}. 

\begin{proof}[Proof of Theorem~\ref{thm:unboundedweights_new}]
Fix $j \in \mathbb{N}$. We have $S_j = \{Q_j^{(1)}, \dots, Q_j^{(l_j)} \}$, where $Q_j^{(k)} = T_j^{(k)} + Q_j$ with $T_j^{(k)} = (0, \dots, 0, \frac{N_j-1}{N_j} |\pi_k(Q_j)|, 0, \dots)$ for each $k \in \{1, \dots, l_j\}$. Since $|Q_{j}^{(k)} \cap Q_{j}| = \frac{1}{N_j} |Q_{j}^{(k)}|$, an application of \eqref{eq:condelta_new} with $E=Q_{j}^{(k)} \cap Q_{j} \subset Q_{j}^{(k)}$ gives
\[
w(Q_{j}^{(k)} \setminus Q_{j}) = w(Q_{j}^{(k)}) - w(Q_{j}^{(k)} \cap Q_{j}) \geq (1-CN_j^{-\delta}) w(Q_{j}^{(k)}).
\]
Similarly, letting $Q_{j,n}^{(k)} := \frac{n}{N_j-1} T_j^{(k)} + Q_j$, $n \in \{0, \dots, N_j-1\}$, we get that $|Q_{j,n}^{(k)} \setminus Q_{j,n+1}^{(k)}| = \frac{1}{N_j}|Q_{j,n}^{(k)}|$. An application of \eqref{eq:condelta_new} with $E=Q_{j,n}^{(k)} \setminus Q_{j,n+1}^{(k)} \subset Q_{j,n}^{(k)}$ gives
\[
\frac{w(Q_{j,n+1}^{(k)})}{w(Q_{j,n}^{(k)})} \geq 1 - \frac{w(Q_{j,n}^{(k)} \setminus Q_{j,n+1}^{(k)})}{w(Q_{j,n}^{(k)})} \geq 1-CN_j^{-\delta}.
\]
Using the above estimates, we obtain
\[
\frac{w(Q_{j}^{(k)} \setminus Q_{j})}{w(Q_{j})} \geq (1-CN_j^{-\delta}) \frac{w(Q_{j,N_j-1}^{(k)})}{w(Q_{j,0}^{(k)})} \geq \big( 1-CN_j^{-\delta} \big)^{N_j} > 0.
\]

Finally, take $f_j = \chi_{Q_j}$. Then for each $x \in Q \in S_j$ we have
\[
M^{\mathcal{R}'} f_j(x) \geq |f_j|_{Q} = \frac{|Q \cap Q_j|}{|Q|} = \frac{1}{N_j}.
\]
Because the sets $Q_j^{(k)} \setminus Q_j$ are disjoint, we obtain
\begin{equation*}
\frac{\|M^{\mathcal{R}'}f_j\|_{L^{q,\infty}(w)}}{\|f_j\|_{L^q(w)}} \geq \frac{1}{2N_j} \, \left(\frac{ \sum_{k=1}^{l_j} w(Q_j^{(k)} \setminus Q_j)}{w(Q_j)}\right)^{1/q}
\geq \frac{\big( 1-CN_j^{-\delta} \big)^{N_j/q} l_j^{1/q}}{2N_j},
\end{equation*} 
which gives the claim.
 \end{proof}
 
 \begin{proof}[Proof of Corollary~\ref{corolweights_new}]
 Observe that, in view of Example~\ref{ex:epsn}, for $N$, the smallest positive integer such that $C N^{-\delta} < 1$ holds, and for each $j \in \mathbb{N}$ one can find a $(\frac{1}{N}, j)$-configuration contained in $\mathcal{R}$. Then, since
 \[
 \lim_{j \rightarrow \infty} \frac{\big( 1-CN^{-\delta} \big)^{N/q} j^{1/q}}{N} = \infty,
 \]
an application of Theorem~\ref{thm:unboundedweights_new} gives the desired result.  
\end{proof}

\begin{proof}[Proof of Corollary~\ref{corolweights_new2}]
The claim follows directly from Corollaries~\ref{cor:RH=Ap}~and~\ref{corolweights_new}.
\end{proof}
The original motivation to investigate various classes of weights is that their elements are usually supposed to behave nicely, particularly when studying the boundedness of the associated maximal operator. Indeed, in case of the Euclidean spaces both, the Muckenhoupt condition and the reverse H\"older condition, say that a given weight is regular enough to share good properties with the trivial weight constantly equal to $1$. 

In our setting the situation is different, as the unweighted maximal operator $M^\mathcal{R}$ is not bounded on $L^q(\tom)$ for any finite $q$. As we mentioned before, ``to be regular'' means ``to share properties with the trivial weight''. Note that the proof  of Theorem~\ref{unweighted_unboundedness} can be adapted to get the proof of Theorem \ref{thm:unboundedweights_new}, and the Muckenhoupt condition was applied to ensure that the argument used in the unweighted setting may be successfully repeated in the weighted setting with only small modifications. This reflects that the problem in the infinite-dimensional setting is not only about the measure, but also about the geometry of the space and the basis.  

Finally, we go back to the discussion after Corollary~\ref{corolweights_new}, regarding the intermediate bases $\mathcal{R}'$ introduced in Theorem~\ref{thm:r_0}. Namely, as for these bases $M^{\mathcal{R}'}$ is bounded on $L^q(\tom)$ in some range of $q$'s, one would ask if this positive result can be transferred to the weighted setting. However, remind that the proof of the positive part of Theorem~\ref{thm:r_0} relies on a very delicate argument which originates in the probability theory. In particular, the fact that the constant weight is invariant under permutations of coordinates is crucial. If a Muckenhoupt weight $w$ is considered instead, then we lose the possibility of using estimates for Bernoulli distributed random variables. So, even if the said transference is indeed possible in some cases, one cannot expect that only small modifications are needed to make the argument used in Theorem~\ref{thm:r_0} valid also in the weighted case. Either way, studying the weight theory for the bases $\mathcal{R}'$ is certainly an interesting direction for further research.

\section{Open questions}\label{sec:reverse}

In the Euclidean setting, when the basis is the family of all cubes in $\mathbb{R}^n$, the corresponding classes given in Definition~\ref{def:weights1} are equivalent to the Fujii--Wilson $A_{\infty}(\mathbb{R}^n)$ class. We would like to examine if the same happens for bases in $\tom$. Before that we provide a suitable definition.
\begin{definition}\label{def:Ainfinity}
Given a basis $\mathcal{B}$ in $\tom$, we say that $w\in A_\infty^\mathcal{B}(\tom)$ if the Fujii--Wilson type $A_\infty^\mathcal{B}(\tom)$ constant  
\begin{equation}\label{Ainfty}
[w]_{A_\infty^\mathcal{B}(\tom)}:=\sup_{Q\in \mathcal{B} }\frac{1}{w(Q)}\int_Q M^\mathcal{B}(w\chi_Q)(x)\, dx
\end{equation}
is finite. 
\end{definition}

Firstly, by the trivial boundedness of the dyadic maximal operator $M^{\mathcal{R}_0}$ on $L^\infty(\tom)$, one obtains its boundedness over all spaces $L^q(\tom)$ with $q \in (1, \infty]$. Thus, in virtue of \cite[Theorem~4.2]{Duoandikoetxea2016}, the classes of Muckenhoupt, reverse H\"older, and $A_\infty^{\mathcal{R}_0}(\tom)$ weights corresponding to the restricted basis $\mathcal{R}_0$ coincide.
In the case of $\mathcal{R}$, in turn, the unboundedness of $M^\mathcal{R}$ leads us to a more delicate study of the corresponding classes of weights, which could end up not being the same.

As a first positive result in the context of $\mathcal{R}$, we note that the local Calder\'on--Zygmund decomposition allows one to prove that the class of $A_\infty^\mathcal{R}(\tom)$ weights is inside the class of reverse H\"older weights. Moreover, we can prove a sharp reverse H\"older inequality in the spirit of those in \cite{Hytoenen2012,Paternostro2019} but in the sharper form obtained in \cite{ParissisRela} adapted to flat weights. This follows from standard arguments and the fact that the cubes in the local Calder\'on--Zygmund decomposition are cubes of $\mathcal{R}$.

\begin{theorem}[{cf.\ \cite[Theorem~1.6]{ParissisRela}}]\label{RH}
Fix $w\in A_\infty^\mathcal{R}(\tom)$ and let $Q$ be a cube in $\mathcal R$. Then for all $r \in \big[1,1+\frac{1}{ [w]_{A_\infty^\mathcal{R}(\tom)}-1}\big)$ we have the reverse H\"older inequality
\[
\frac{1}{|Q|}\int_Q w^r(x)\,dx \leq C_{[w]_{A_\infty^\mathcal{R}(\tom)},r}\left(\frac{1}{|Q|}\int_Q w(x)\,dx \right)^r
\]
with
\begin{equation*}
C_{[w]_{A_\infty^\mathcal{R}(\tom)},r}=[w]_{A_\infty^\mathcal{R}(\tom)}\frac{r'-1}{r'-1-2([w]_{A_\infty^\mathcal{R}(\tom)}-1)}.
\end{equation*}
\end{theorem} 

\noindent As an immediate corollary we have the following.

\begin{corollary}\label{cor:AinftyinRH}
Each $A_\infty^\mathcal{R}(\tom)$ weight is a reverse H\"older weight with respect to $\mathcal{R}$.
\end{corollary}

In the Euclidean setting, the corresponding counterpart of the sharp reverse H\"older inequality stated in Theorem~\ref{RH} is useful when proving the so-called Buckley's theorem (see  \cite[Theorem~2.5]{TesisBuckley}), which provides a sharp weighted bound for the Hardy--Littlewood maximal operator in terms of the Muckenhoupt constant of the corresponding weight. In our context, by standard arguments (or by combining Corollaries~\ref{cor:RH=Ap} and \ref{cor:AinftyinRH}) we get that $A_\infty^\mathcal{R}(\tom)$ weights are Muckenhoupt weights.

\begin{corollary}\label{cor:AinftyinAp}
Each $A_\infty^\mathcal{R}(\tom)$ weight is a Muckenhoupt weight with respect to $\mathcal{R}$.
\end{corollary} 

Nevertheless, in contrast with the Euclidean finite-dimensional case, it is not immediate that Muckenhoupt weights and reverse H\"older weights with respect to $\mathcal{R}$ are $A_\infty^\mathcal{R}(\tom)$ weights (recall that the two former classes are indeed equivalent, see Corollary~\ref{cor:RH=Ap}). 
To sum up, we know that the following relations hold:
\begin{equation}
\label{eq:RHM}
 A_\infty^\mathcal{R}(\tom)\subseteq \bigcup_{p \in (1, \infty)}A_p^\mathcal{R}(\tom) =   \bigcup_{r \in (1, \infty)}\mathrm{RH}_r^\mathcal{R}(\tom).
\end{equation}
However, we have not been able to prove or disprove the missing inclusion. This motivates us to ask the following question.

\begin{question}\label{q:classes}
Is it true that  
$A_\infty^\mathcal{R}(\tom)= \bigcup_{p\in (1, \infty)}A_p^\mathcal{R}(\tom)$? In view of \eqref{eq:RHM} this would be equivalent to prove that $A_\infty^\mathcal{R}(\tom) = \bigcup_{r\in (1, \infty)}\mathrm{RH}_r^\mathcal{R}(\tom)$.
\end{question}

\begin{remark}\label{genuine}
Using standard tools one can show that weights $v(x) = \prod_{j=1}^n w_j(x_j)$, where each $w_j$ belongs to $\bigcup_{p\in (1, \infty)}A_p(\mathbb{T})$, are indeed $A_\infty(\tom)$ weights. Thus a counterexample for Question~\ref{q:classes}, if any exists, probably has more complicated structure. This in particular motivates the seek for ``genuine'' weights on $\tom$. \end{remark}

We also point out that for $p=1$ the inclusion $A_1^\mathcal{R}(\tom) \subset A_\infty^\mathcal{R}(\tom)$ holds. Indeed, consider an $A_1^\mathcal{R}(\tom)$ weight $w$, meaning that for a.e. $x \in \tom$ we have
 \begin{equation*}
 Mw(x)\le [w]_{A_1^\mathcal{R}(\tom)} w(x).   
 \end{equation*}
 Then, we can compute the $A_\infty^\mathcal{R}(\tom)$ constant by noting that, for any cube $Q\in\mathcal{R}$, 
 \begin{equation*}
 \frac{1}{w(Q)}\int_Q M(w\chi_Q)(x)\ dx\le \frac{1}{w(Q)}\int_Q [w]_{A_1^\mathcal{R}(\tom)} w(x)\ dx\le [w]_{A_1^\mathcal{R}(\tom)}.
 \end{equation*}
 This gives us the inequality
 \begin{equation}\label{eq:Ainfty-LE-A1}
 [w]_{A_\infty^\mathcal{R}(\tom)}\le [w]_{A_1^\mathcal{R}(\tom)}
 \end{equation}
 and, as a consequence, the postulated inclusion.

Our second question is in line with the discussion after Corollary~\ref{corolweights_new}.

\begin{question}\label{q:intermediate}
 Given $q_0 \in [1, \infty)$, consider the intermediate bases $\mathcal{R}'$ used in the proof of Corollary~\ref{corollary2}. Is $A_p^{\mathcal{R}'}(\tom)$ the correct class to characterize the weighted boundedness of $M^{\mathcal{R}'}$ on $L^p(w)$ for $p \in (q_0, \infty)$? 
 \end{question} 

To close the paper, we also raise the following question.

\begin{question}
Are there any bases $\mathcal{R}'$ for which the behavior of the corresponding maximal function is significantly different in the unweighted and weighted settings? 
\end{question}

  \section*{Acknowledgements}

The authors would like to thank Emilio Fern\'andez for valuable discussions.

The first, second and fifth authors are supported by the Basque Government through the BERC 2018-2021 program and by the Spanish State Research Agency through BCAM Severo Ochoa excellence accreditation SEV-2017-2018. The fifth author also acknowledges the project PID2020-113156GB-I00, the RyC project RYC2018-025477-I, and Ikerbasque. The third and fourth authors acknowledge the projects PICT-2019--03968, PICT 2018-3399 and UBACyT 20020190200230BA. The third author is also partially supported by PICT-2016-2616 (Joven). 

\bibliographystyle{amsalpha}

%\bibliographystyle{amsplain}

%\bibliography{References}
\newcommand{\etalchar}[1]{$^{#1}$}
\providecommand{\bysame}{\leavevmode\hbox to3em{\hrulefill}\thinspace}
\providecommand{\MR}{\relax\ifhmode\unskip\space\fi MR }
\providecommand{\MRhref}[2]{%
  \href{http://www.ams.org/mathscinet-getitem?mr=#1}{#2}
}
\providecommand{\href}[2]{#2}

\end{document}